%% file: RiccatiVolterraFinal.tex
\documentclass[11pt]{article}

\usepackage{amssymb,amsmath,amsfonts}
\usepackage{graphicx,color,enumitem}
\usepackage{mathrsfs}  
\usepackage{amsthm} 
\usepackage{bm}
\RequirePackage[numbers]{natbib}

\usepackage{stmaryrd}

\usepackage{geometry}
 \usepackage{mathtools}
 \mathtoolsset{showonlyrefs}

\usepackage[colorinlistoftodos, textwidth=4cm, shadow]{todonotes}
\RequirePackage[colorlinks,citecolor=blue,urlcolor=blue]{hyperref}

\usepackage{bbm}
\usepackage{soul}

\newcommand{\ind}{\mathbbm{1}}

\newcommand{\E}{{\mathbb E}}
\newcommand{\F}{{\mathbb F}}

\renewcommand{\P}{{\mathbb P}}

\newcommand{\Uc}{{\mathcal U}}
\newcommand{\Vc}{{\mathcal V}}

\newcommand{\R}{{\mathbb R}}
\renewcommand{\S}{{\mathbb S}}
\newcommand{\N}{{\mathbb N}}

\newcommand{\Bcal}{{\mathcal B}}

\newcommand{\Ncal}{{\mathcal N}}

\newcommand{\Tcal}{{\mathcal T}}

\newcommand{\Fc}{{\mathcal F}}

\DeclareMathOperator{\tr}{tr}

\DeclareMathOperator{\esssup}{ess\,sup}

\newtheorem{theorem}{Theorem}

\newtheorem{definition}[theorem]{Definition}

\newtheorem{lemma}[theorem]{Lemma}

\theoremstyle{remark}
\newtheorem{remark}[theorem]{Remark}

\numberwithin{equation}{section}
\numberwithin{theorem}{section}

\newcommand{\red}{\color{red}}

\definecolor{darkgreen}{rgb}{0,0.7,0}
\newcommand{\green}{\color{darkgreen}}

\newcommand{\iii}{{\vert\kern-0.25ex\vert\kern-0.25ex\vert}}

\newcommand{\T}{\top}

\newcommand{\ints}{\int_{\R_+}}

\renewcommand{\c}{\alpha}

\def \ep{\hbox{ }\hfill$\Box$}

\begin{document}

\input{./functions.tex}

\title{Integral operator Riccati  equations arising in stochastic Volterra control problems}

\author{Eduardo ABI JABER \footnote{Universit\'e Paris 1 Panth\'eon-Sorbonne, Centre d'Economie de la Sorbonne, 106, Boulevard de l'H\^opital, 75013 Paris, \sf  eduardo.abi-jaber at univ-paris1.fr.}
\quad\quad Enzo MILLER \footnote{LPSM, Universit\'e de Paris, Building Sophie Germain, Avenue de France, 75013 Paris,  \sf  enzo.miller at polytechnique.org}
\quad\quad  Huy\^en PHAM \footnote{LPSM, Universit\'e de Paris,  Building Sophie Germain, Avenue de France, 75013 Paris, \sf pham at lpsm.paris 
The work of this author  is supported by FiME (Finance for Energy Market Research Centre) and the ``Finance et D\'eveloppement Durable - Approches Quantitatives'' EDF - CACIB Chair.
 }
}
\maketitle

\begin{abstract}
We establish existence and uniqueness for infinite dimensional Riccati equations taking values in the Banach space $L^1(\mu \otimes \mu)$ for certain signed matrix measures $\mu$ which are not necessarily finite. Such equations can be seen as the infinite dimensional analogue of matrix Riccati equations and they appear in the Linear--Quadratic control theory of stochastic Volterra equations.
\end{abstract}

\vspace{5mm}


\vspace{5mm}


\newpage


\section{Introduction}
Fix $d,d',m \in \mathbb N$ and $\mu$ a $d\times d'$-matrix signed measure $\mu$. This paper deals with the infinite dimensional Backward Riccati equation
\begin{equation} \label{formalkernelRiccati}
\left\{
\begin{array}{ccl}
{\Gamma}_{T}(\theta,\tau) &=& {0}  \\
\dot    {{\Gamma}}_{t}(\theta,\tau) &=&  (\theta+\tau)\Gamma_{t}(\theta,\tau)  - Q   - D^\top \int_{\R_+^2} \mu(d\theta')^\top \Gamma_{t}(\theta',\tau')\mu(d\tau') D  \\
& & \; - \;  {B}^\top \int_{\R_+} \mu(d\theta')^\top \Gamma_{t}(\theta',\tau)  -   \int_{\R_+}  \Gamma_{t}(\theta,\tau') \mu(d\tau')B  +  \;  S_{t}(\theta)^\T \hat{N}_{t}^{-1} S_{t}(\tau),  
\end{array}
\right.
\end{equation}
where
\begin{align*}
S_{t}(\tau)  &= C^\T \ints \mu(d\theta)^\top \Gamma_{t}(\theta,\tau) + F^\top \int_{\R_+^2}  \mu(d\theta')^\top \Gamma_{t}(\theta',\tau')  \mu(d\tau') D, \\
\hat{N}_{t}  &= N + F^\T \int_{\R_+^2} \mu(d\theta)^\T \Gamma_{t}(\theta,\tau)   \mu(d\tau)  F,
\end{align*}
and  $B,D$ $\in$ $\R^{d'\times d}$, $C,F$ $\in$ $\R^{d'\times m}$, $Q$ $\in$ $\R^{d\times d}$ and $N$ $\in$ $\R^{m\times m}$, and {$^\top$ is the transpose operation.}    
Here $\mu$ is not necessarily finite and satisfies
\begin{align}\label{eq:mucond}
\int_{\R_+} \big( 1 \wedge \theta^{-1/2}\big) |\mu|(d\theta) <\infty,
\end{align}
where $|\mu|$ denotes the total variation of $\mu$. We look for solutions $\Gamma: [0,T]\times \R_+^2 \to \R^{d\times d}  $ with values in $L^1(\mu\otimes \mu)$ {(see the precise definition in Section \ref{secmain})  to ensure that  equation \eqref{formalkernelRiccati} is well-posed.}

In particular, if $d=d'=1$ and $\mu(d\theta)=\sum_{i=1}^n \delta_{\theta^n_i}(d\theta)$, \eqref{formalkernelRiccati} reduces to a  $n\times n$-matrix Riccati equation  for 
$\Gamma^n=(\Gamma(\theta_i^n,\theta_j^n))_{1\leq i,j\leq n}$, only written componentwise. Such matrix Riccati equation appears in finite dimensional Linear-Quadratic (LQ) control theory,  see e.g.   \cite[chapter 7]{yong1999stochastic}. (One could also recover $d\times d$-matrix Riccati equation by setting $d=d'$ and $\mu=I_d \delta_0$.)

For more general measures $\mu$, e.g.~with infinite support,   \eqref{formalkernelRiccati} can be seen as the infinite-dimensional extension of matrix Riccati equations and one could expect a connection with LQ control in infinite dimension. This is indeed the case, and our motivation for studying the Riccati equation \eqref{formalkernelRiccati} comes from 
infinite dimensional lifts of LQ control theory of non-Markovian stochastic Volterra equations. Setting 
$$  K(t)=\int_{\R_+} e^{-\theta t} \mu(d\theta), \quad t >0,$$
one can consider the controlled $d$-dimensional  stochastic linear Volterra equation 
\begin{align} 
X_t &=  \int_0^t K(t-s) \big( BX_s + C \alpha_s\big) ds + \int_0^t K(t-s) \big( D X_s + F \alpha_s \big) dW_s, 
\end{align}
where $W$ is a one dimensional Brownian motion and $\alpha$ is a suitable control taking its values in $\R^m$. Observe that the integrability condition on the measure $\mu$ allows 
singularity of the kernel $K$ at $0$, and includes the case of a fractional kernel $K_H(t)=t^{H-1/2}$ with Hurst parameter $H \in (0,1/2)$  with a corresponding measure $\mu(d\theta)=c_H\theta^{-H-1/2}d\theta$, for some normalizing constant $c_H$. 
The linear-quadratic  control problem consisting in the minimization over $\alpha$ of the cost functional 
\begin{align}  
J(\alpha) &= \E \Big[ \int_0^T\big(  X_t^\top Q X_t + \alpha_t^\top N \alpha_t \big) dt \Big], 
\end{align}
can be  explicitly  solved using the Riccati equation \eqref{formalkernelRiccati}, see \cite{abietal19a}. 

When  $D=F=0$, 
the Riccati equation  \eqref{formalkernelRiccati} also enters in the computation of the Laplace transform of  $\tr\left(\int_0^T Z_s^\top Q Z_s ds \right)$, where  $Z$ is the $d\times n$-matrix valued Gaussian process
$$  Z_t = Z_0 + \int_0^t K(t-s) BZ_s ds + \int_0^t K(t-s)Cd\widetilde W_s, \quad t\geq 0,$$
and $\widetilde W$ is a $m\times n$ matrix Brownian motion, see \cite{AJwishart}. 

The Riccati equation \eqref{formalkernelRiccati} can be also connected to an operator Riccati equation as follows.  
Denote by $L^1(\mu)$ the Banach space of $\mu$-a.e. equivalence classes of $|\mu|$-integrable functions 
$\varphi$ $:$ $\R_+$ $\rightarrow$ $\R^{d'}$ endowed with the norm $\|\varphi\|_{L^1(\mu)}$ $=$ $\int_{\R_+} |\mu|(d\theta) |\varphi(\theta)|$,  by  
$L^\infty(\mu)$ the space of  measurable functions from $\R_+\to\R^{d'}$, which are bounded $\mu$-a.e., and   introduce the dual pairing:
\begin{eqnarray} \label{dual}
\langle \varphi,  \psi \rangle_{\mu} &:=& \int_{\R_+} \varphi(\theta)^\top \mu(d\theta)^\top \psi(\theta), 
\quad  (\varphi,\psi)\in L^1(\mu)\times L^{\infty}(\mu^\top).   
\end{eqnarray}
Given any  bounded kernel  solution $\Gamma$ to \eqref{formalkernelRiccati}, let us consider the corresponding  linear integral  operator 
{$\boldsymbol{\Gamma}$ $:$ $[0,T]\times L^1(\mu)$ $\rightarrow$ $L^\infty(\mu^\top)$ defined by} 
\begin{align*}
(\boldsymbol{\Gamma}_t \varphi)(\theta) &= \int_{\R_+} \Gamma_t(\theta,\tau) \mu(d\tau) \varphi(\tau), \quad t \in [0,T], \;\; \varphi \in L^1(\mu). 
\end{align*} 
It is then straightforward to see  that $\boldsymbol{\Gamma}$ solves  the operator Riccati equation on $L^1(\mu)$:  
\begin{equation} \label{Riccatioperator}
\left\{
\begin{array}{ccl}
\boldsymbol{\Gamma}_{T} &=& \boldsymbol{0} \\
\dot    {\boldsymbol{\Gamma}}_{t}&=& - \boldsymbol{\Gamma}_{t}A^{mr} - \left(\boldsymbol{\Gamma}_{t}A^{mr}\right)^{*}-  \boldsymbol{Q} - \boldsymbol{D}^* \boldsymbol{\Gamma}_{t} \boldsymbol{D}   - \boldsymbol{B}^* \boldsymbol{\Gamma}_{t}  - \left(\boldsymbol{B^*}\boldsymbol{\Gamma}_{t}\right)^*  \\
& &  \;  + \;  \left({C}^* \boldsymbol{\Gamma}_{t} + {F}^* \boldsymbol{\Gamma}_{t} \boldsymbol{D} \right)^*\left( {N} + {F}^*\boldsymbol{\Gamma}_{t} {F} \right)^{-1} 
\left({C}^* \boldsymbol{\Gamma}_{t}  +{F}^* \boldsymbol{\Gamma}_{t} \boldsymbol{D} \right), \quad t \in [0,T],     
\end{array}
\right.
\end{equation}
where  $A^{mr}$ is the mean-reverting operator  acting on measurable functions $\varphi \in L^{1}(\mu)$ by
\begin{align*}
(A^{mr}\varphi)(\theta) &=-\theta \varphi(\theta), \quad \theta \in \R_+,
\end{align*}
$\boldsymbol{B}$, $\boldsymbol{D}$ are the  integral operators on  $L^1(\mu)$ (defined  similarly to $\boldsymbol{\Gamma}$) induced by the constant matrices $B$, $D$,  and 
by misuse of notation, $C$, $F$ denote the respective constant operators on $\R^m$ induced by the matrices $C$, $F$:
\begin{align*}
(C a)(\theta) &= C a, \quad (F a) (\theta) \; = F a, \quad \theta  \in \R_+,\quad a  \in \R^m. 
\end{align*}
Here the symbol $^*$ denotes the adjoint operation  with respect to the dual pairing. 
The last equation \eqref{Riccatioperator} is more in line with  the formulation of operator Riccati equations appearing in LQ control theory in Hilbert or Banach spaces, 
see  \cite{curtain}, \cite{daprato}, \cite{flandoli1986direct},  \cite{guates05}, \cite{lasiecka}, \cite{kos}, \cite{hu2018stochastic}, \cite{arta}.     

Let us also mention that a related infinite-dimensional Riccati equation appeared in \cite{alfonsi2013capacitary} for the minimization problem of an energy functional defined in terms of a non-singular (i.e. $K(0)$ $<$ $\infty$) completely monotone kernel.

 The main contribution of the  paper is to establish the  existence and uniqueness  of a   solution to the  kernel Riccati equation \eqref{formalkernelRiccati}.
  The aforementioned results on the solvability of  Riccati equations in infinite  dimensional spaces cannot be directly applied in our setting for two reasons. First, they are valid for Hilbert and reflexive Banach spaces,  while  $L^1(\mu)$ is in general not reflexive, unless $\mu$ has finite support, and mostly apply to the cases without multiplicative noise, i.e., $\boldsymbol{D}$ $=$ $0$, and 
without control on the diffusion coefficient, i.e. $\boldsymbol{F}$ $=$ $0$, with the noticeable exception in \cite{hu2018stochastic}. 
  Second, they concern the operator Riccati equation \eqref{Riccatioperator}, which is not enough for our purposes, as we still need to argue that  $\boldsymbol{\Gamma}$  is an integral operator induced by some bounded symmetric kernel function  $\Gamma$ satisfying \eqref{formalkernelRiccati}. 
We will therefore work directly on the level of  the kernel Riccati equation  \eqref{formalkernelRiccati} (which will also be referred to as integral operator Riccati equation) by adapting 
the technique  used  in classical finite-dimensional linear-quadratic control  theory \citep[Theorem 6.7.2]{yong1999stochastic} 
with the follo\-wing steps: (i) we first construct  a sequence of Lyapunov solutions 
$(\Gamma^i)_{i\geq 0}$  by successive iterations, (ii) we then show the convergence of  $(\Gamma^i)_{i\geq 0}$  in $L^1(\mu\otimes \mu)$, (iii) we next prove 
that the limiting point is a solution to the Riccati equation \eqref{formalkernelRiccati}, (iv) we finally prove the continuity and uniqueness for the Riccati solution. We stress that such method has already been applied to prove the existence of operator  Riccati equations of the form \eqref{Riccatioperator} in Hilbert spaces (see \cite{hu2018stochastic}) and in reflexive Banach spaces (see \cite{arta}). However, for the kernel Riccati equation \eqref{formalkernelRiccati},  the proof is more intricate. The reason is that we need to establish the convergence of the kernels $(\Gamma^i)_{i\geq 0}$ which is a stronger requirement than the usual convergence of the operators $(\boldsymbol{\Gamma}^i)_{i \geq 0}$. As a consequence, we obtain that the sequence of integral operators $(\boldsymbol{\Gamma}^i)_{i \geq 0}$ converges to some limit which is also an integral operator.

The paper is organized as follows.  We formulate precisely  our  main result in Section \ref{secmain}.  Section \ref{secLya} establishes the existence of a solution to an infinite dimensional  Lyapunov equation. Section \ref{secRiccati} is  devoted to the solvability of the Riccati equation. Finally, we collect in the Appendix some useful results.


\section{Preliminaries and main result} \label{secmain}


Let us first  introduce some notations that will be used in the sequel of the paper. 
For any $d_{11}\times d_{12}$-matrix valued measure $\mu_1$,  and $d_{21} \times d_{22}$-matrix valued measure $\mu_2$ on $\R_+$,  
the Banach space  $L^1(\mu_1 \otimes \mu_2)$ consists of $\mu_1\otimes \mu_2$-a.e.~equivalence classes of 
$|\mu_1|\otimes |\mu_2|$-integrable functions $\Phi:\R_+^2 \to \R^{d_{11}\times d_{21}}$  endowed  with the norm 
$\|\Phi\|_{L^1(\mu_1 \otimes \mu_2)}=\int_{\R^2_+} {|\mu_1|}(d\theta) |\Phi(\theta,\tau)| {|\mu_2|}(d\theta) < \infty.$ 
For any such $\Phi$, the integral 
$$ \int_{\R_+^2} \mu_1(d\theta)^\top \Phi(\theta,\tau) \mu_2(d\tau) $$ 
is well defined by virtue of \cite[Theorem 5.6]{GLS:90}. 
We also denote 
by  $L^\infty(\mu_1\otimes\mu_2)$ the set of  measurable functions $\Phi:\R_+^2 \to \R^{d_{11}\times d_{21}}$, which are bounded $\mu_1\otimes \mu_2$-a.e. 

\vspace{1mm}

We shall prove the existence of a nonnegative symmetric kernel {solution} to the Riccati equation \eqref{formalkernelRiccati} in the following sense.

\begin{definition}\label{D:nonnegative}
	Let $\Gamma:\R_+^2\to \R^{d\times d}$ such  that $\Gamma \in L^{\infty}(\mu \otimes \mu)$.  We say that $\Gamma$ is symmetric if
	\begin{align*}
	\Gamma(\theta,\tau) &= \; \Gamma(\tau,\theta)^\top, \quad \mu\otimes \mu-a.e.
	\end{align*}  
	and  nonnegative if 
	\begin{align*}
	\int_{\R_+^2} \varphi(\theta)^\top \mu(d\theta)^\top	\Gamma(\theta,\tau) \mu(d\tau) \varphi(\tau) &\geq \; 0, \quad \mbox{ for all }   \varphi \in L^{1}(\mu). 
	\end{align*} 
	We denote by $\S_+^d(\mu \otimes \mu)$ the set of all  symmetric and nonnegative  $\Gamma \in L^{\infty}(\mu \otimes \mu)$, 
	and we define on  $\S_+^d(\mu \otimes \mu)$ the  partial order   relation 
	$\Gamma^1 \succeq_{\mu} \Gamma^2$ whenever  $(\Gamma^1-\Gamma^2) \in \S^d_+(\mu \otimes \mu)$. 
\end{definition}

{\begin{remark} (On utilise cette notation dans le lemme 4.1(ii)) 
$\S_+^d(\delta_0 \otimes \delta_0)$ reduces to $\S^d_+$, the cone of  symmetric semidefinite $d\times d$--matrices.
 \ep
\end{remark}
}

\vspace{1mm}

Given a kernel  $\Gamma$, we define the integral  operator $\boldsymbol{\Gamma}$ by
\begin{align} \label{defintegral}
(\boldsymbol{\Gamma} \varphi)(\theta) &= \int_{\R_+} \Gamma(\theta,\tau) \mu(d\tau) \varphi(\tau). 
\end{align}
Notice that when $\Gamma$ $\in$ $L^1(\mu\otimes\mu)$, the operator $\boldsymbol{\Gamma}$ is well-defined on $L^\infty(\mu)$,  
and we have  $\boldsymbol{\Gamma}\varphi$ $\in$ $L^1(\mu^\top)$, for $\varphi$ $\in$ $L^\infty(\mu)$. In this case $\langle \boldsymbol{\Gamma}\varphi ,\psi  \rangle_{\mu^\top}$ is well defined for all $\varphi,\psi \in L^\infty(\mu)$. Moreover,  
when  $\Gamma \in  L^{\infty}(\mu\otimes \mu)$,  the operator  $\boldsymbol{\Gamma}$ is well-defined on $L^1(\mu)$, and we have $\boldsymbol{\Gamma}\varphi$ $\in$ 
$L^\infty(\mu^\top)$ for  $\varphi$ $\in$ $L^1(\mu)$.  In this case, $\langle \varphi, \boldsymbol{\Gamma}\psi \rangle_{\mu}$ is well defined for all $\varphi,\psi \in L^1(\mu)$.

Whenever  $\Gamma$ $\in$ $L^{\infty}(\mu \otimes \mu)$ is  a symmetric kernel, we have
\begin{align*}
\langle \varphi,  \boldsymbol{\Gamma} \psi \rangle_{\mu}, &= \; \langle   \psi ,  \boldsymbol{\Gamma}\varphi \rangle_{\mu},  
\quad \varphi, \psi \in L^1(\mu), 
\end{align*}
and $\boldsymbol{\Gamma}$ is said to be symmetric. For $\Gamma \in \S^d_+(\mu\otimes \mu)$, the nonnegativity reads
\begin{align*}
\langle \varphi,  \boldsymbol{\Gamma} \varphi \rangle_{\mu} & \geq \; 0, \quad \forall \varphi \in L^{1}(\mu). 
\end{align*}
The kernel Riccati equation \eqref{formalkernelRiccati} can be compactly written in the form 
\begin{equation} \label{eq:Riccatis_monotone}
\begin{array}{rclc}
 \dot{\Gamma}_{t}(\theta,\tau)  &=&  (\theta+\tau)\Gamma_{t}(\theta,\tau) -  \mathcal {R} (\Gamma_{t})(\theta,\tau),  & \quad  \Gamma_{T}(\theta, \tau) = 0  
\end{array}
\end{equation}
where we define
\begin{align}
\mathcal {R}(\Gamma) (\theta,\tau)&= \; Q +D^\top \int_{\R_+^2} \mu(d\theta')^\top \Gamma(\theta',\tau') \mu(d\tau') D
 + B^\top \int_{\R_+} \mu(d\theta')^\top \Gamma(\theta', \tau)  \\  
 &\quad\quad  + \;  \ints \Gamma(\theta, \tau') \mu(d \tau') B   - S(\Gamma)(\theta)^\top \hat{N}^{-1}(\Gamma) S(\Gamma)(\tau)  \label{eq:R1} 
\end{align}
with 
\begin{equation}  \label{eq:hatN}
\left\{
\begin{array}{rl}
S(\Gamma)(\tau)  & = \;   C^\T \ints \mu(d\theta)^\top \Gamma(\theta,\tau)  + F^\top \int_{\R_+^2}  \mu(d\theta' )^\top
\Gamma(\theta',\tau')  \mu(d\tau') D \\
\hat{N}(\Gamma) & = \;  N + F^\T \int_{\R_+^2} \mu(d\theta)^\T \Gamma(\theta,\tau)   \mu(d\tau)  F.  
\end{array}
\right.
\end{equation}

\vspace{1mm}

The  following definition specifies the concept of  solution to the kernel Riccati equation  \eqref{eq:Riccatis_monotone}.

\begin{definition}
By a solution to the  kernel Riccati equation \eqref{eq:Riccatis_monotone}, we mean a function $\Gamma$ $\in$ $C([0,T],L^1(\mu\otimes \mu))$  such that  
\begin{align}
 \Gamma_{t}(\theta,\tau)  &= \;   \int_t^T e^{-(\theta+\tau)(s-t)}     \mathcal {R} (\Gamma_{s})(\theta,\tau) d  s,  &&  0\leq t\leq T, \quad  \mu\otimes \mu-a.e.
\label{eq:Riccati_monotone_kernel_mild}
\end{align}
where $\mathcal R$  is defined  by \eqref{eq:R1}.   In particular  $\hat N(\Gamma_{t})$ given by \eqref{eq:hatN} is  invertible for all $t\leq T$. 
\end{definition}

\vspace{2mm}

Our main result is stated as follows. 
\begin{theorem}\label{T:Riccatiexistence}  Let $\mu$ be a $d\times d'$-signed matrix measure satisfying \eqref{eq:mucond}. 
	Assume that  
	\begin{align}
	\label{assumption:QN}
	Q \in \S^d_+ , \quad  N-\lambda I_m \in  \S^m_+,
	\end{align}
	for some $\lambda>0$. Then, there exists a unique  solution $\Gamma$  $\in$ $C([0,T],L^1(\mu\otimes \mu))$ 
	to the kernel Riccati equation \eqref{eq:Riccatis_monotone} such that $\Gamma_t \in \S^d_+(\mu\otimes \mu)$, for all $t\leq T$.  
	Furthermore,   there exists some positive constant $M$ such that  
	\begin{align}
	\label{eq:estimateRiccati}
	\int_{\R_+} |\mu|(d\tau) |\Gamma_{t}(\theta,\tau)| & \leq \;  M, \quad \mu-a.e., \quad 0 \leq t\leq T.
	\end{align}
\end{theorem}

The rest of the paper is dedicated to the proof  of Theorem~\ref{T:Riccatiexistence}.	Lemmas~\ref{prop:existence_riccati} and \ref{L:Riccatiestimate} provide the existence of a solution in $C([0,T],L^1(\mu\otimes \mu))$ such that $\Gamma_t \in \S^d_+(\mu\otimes \mu)$, for all $t\leq T$. The uniqueness statement is established in Lemma~\ref{l:unicity_riccati}.

\section{Infinite dimensional  Lyapunov equation} \label{secLya} 

Fix $d_{11},d_{12},d_{21},d_{22} \in \mathbb N$. For each $i=1,2$, we  let  $\mu_i$ be a  $d_{i1} \times d_{i2}$-matrix valued measure on $\R_+$, and we define the scalar kernel
\begin{align}
\label{eq:Kbar}
\bar K_i(t)= \int_{\R_+} e^{-\theta t} |\mu_i|(d\theta), \quad t>0,
\end{align}
which is in $L^2([0,T],\R)$ provided that $\mu_i$ satisfies  $\int_{\R_+} \big( 1 \wedge \theta^{-1/2}\big) |\mu_i|(d\theta)$ $<$ $\infty$ (which we shall assume in this section), 
see  \cite[Lemma A.1]{abietal19a}. 

\vspace{1mm}

We first establish the existence and uniqueness for the following infinite dimensional Lyapunov equation: 

\begin{align}
\Psi_t(\theta, \tau) &=  \int_t^T e^{-(\theta + \tau)(s-t)} F(s,\Psi_s)(\theta,\tau)ds, \quad { t\leq T, \quad  \mu_1\otimes\mu_2}-a.e. 	\label{def:infinite_dim_lyapunov1} 
\end{align}
where
\begin{align}
F(s,\Psi)(\theta, \tau)&= \tilde Q_s(\theta, \tau) + \tilde D^1_s(\theta)^\T \int_{\R_+^2} { \mu_1(d\theta')^\T \Psi(\theta',\tau ') \mu_2(d\tau ')} \tilde D^2_s(\tau) \\
&\quad  + \tilde B^1_s(\theta)^\T \int_{\R_+} \mu_1(d\theta')^\T \Psi(\theta',\tau)  + \int_{\R_+} \Psi(\theta,\tau ') \mu_2(d\tau ') \tilde B^2_s(\tau), 	\label{def:infinite_dim_lyapunov2}
\end{align}
for some coefficients $\tilde Q,\tilde B^1,\tilde B^2,\tilde D^1,\tilde D^2$ satisfying suitable assumptions made precise  in the following theorem.  

\begin{theorem}
	\label{thm:existence_unicite_lyapunov}
    Let  $\tilde Q:[0,T]\times \R_+^2 \to \R^{d_{11}\times d_{21}}$ be a measurable function and for each $i=1,2$, $\tilde B^i,\tilde D^i: [0,T]\times\R_+ \to \R^{d_{i2}\times d_{i1}} $ be two measurable  functions. Assume that  there exists $\kappa>0$ such that
    	\begin{align}
        	\label{assumption:lyapunov_ODE}
            |\tilde Q_s(\theta,\tau)| +	\sum_{i=1}^2 |\tilde B^i_s(\theta)| + |\tilde D^i_s(\theta)|^{1/2} &\leq \kappa, \quad  dt\otimes \mu_1 \otimes \mu_2-a.e.
    	\end{align}
    	Then, there exists a unique   solution $\Psi \in C([0,T],L^1(\mu_1 \otimes \mu_2))$  to \eqref{def:infinite_dim_lyapunov1}-\eqref{def:infinite_dim_lyapunov2}. In particular, 
    	\begin{align}
        	\label{eq:estimatelyap1}
        	\sup_{t \leq T} \|\Psi_{t}\|_{L^1(\mu_1 \otimes \mu_2)} < \infty. 
    	\end{align}
    	Furthermore, there exists a constant $\kappa'>0$ such that 
    	\begin{align}
        	\int_{\R_+} |\mu_1|(d\theta)|\Psi_t(\theta,\tau)| \leq \kappa',\quad \mu_2-a.e., \quad t\leq T,  \label{eq:estimatelyap2} \\
        	\int_{\R_+} |\mu_2|(d\tau)|\Psi_t(\theta,\tau)| \leq \kappa',\quad \mu_1-a.e., \quad t\leq T.  \label{eq:estimatelyap3}
    \end{align}
\end{theorem}

\begin{remark}\label{R:Linfty}
		Since the solution $\Psi$ satisfies \eqref{eq:estimatelyap1}, \eqref{eq:estimatelyap2} and \eqref{eq:estimatelyap3},  it follows that
		\begin{align}
		|F(\Psi)(\theta,\tau)| \leq  c, \quad dt \otimes \mu \otimes \mu-a.e.
		\end{align}
		for some constant $c$. Combined with \eqref{def:infinite_dim_lyapunov1}, one gets that $\Psi_t$ $\in$  $L^{\infty}(\mu_1 \otimes \mu_2)$, for all $t\leq T$. 
		\ep
\end{remark}

\vspace{1mm}

The proof of Theorem~\ref{thm:existence_unicite_lyapunov}  follows from the three following lemmas.

\begin{lemma}\label{L:lyap_ST}
	Under the assumptions of Theorem~\ref{thm:existence_unicite_lyapunov}, there exists a unique $L^1(\mu_1\otimes \mu_2)$-valued function $t\in [0,T]$ $\to$ $\Psi_t$ satisfying  \eqref{eq:estimatelyap1},  and such that    \eqref{def:infinite_dim_lyapunov1}-\eqref{def:infinite_dim_lyapunov2} hold. 
\end{lemma}

\begin{proof}
		The proof is an application of the  contraction mapping principle. We denote by $\mathcal B_T$ the space of  measurable and bounded functions $\Psi:[0,T] \to L^1(\mu_1 \otimes \mu_2)$  endowed with the norm 
		\begin{align*}
		\| \Psi\|_{\mathcal B_T}:= \sup_{t\leq T}\| \Psi_t \|_{L^1(\mu_1\otimes \mu_2)}  <\infty.
		\end{align*}
		The space $(\mathcal B_T,\|\cdot\|_{\mathcal B_T})$ is a Banach space. We consider the following family of norms on $\mathcal B_T$:
	\[
	\|\Psi\|_{\lambda} :=  \sup_{t \leq T} e^{-\lambda (T-t))} \|\Psi_{t}\|_{L^1(\mu_1 \otimes \mu_2)}, \quad \lambda >0.
	\]
	For every $\Psi \in \Bcal_T$, define a new function $t \mapsto (\Tcal \Psi)_t$ by
	\bes{
	(\Tcal{\Psi})_t(\theta, \tau) &= \int_t^T e^{-(\theta + \tau)(s-t)} F(s,\Psi_s)(\theta,\tau)ds, \quad \mu_1\otimes \mu_2-a.e.,
}
	where $F$ is given by \eqref{def:infinite_dim_lyapunov2}.  Since the norms $\|\cdot\|_{\mathcal B_T}$ and $\|\cdot\|_{\lambda}$ are equivalent, it is enough to find $\lambda>0$ such that $\mathcal T$ defines a contraction on $(\mathcal B_{T},\|\cdot\|_{\lambda})$. We thus look for $\lambda>0$ and $M<1$ such that 
	\begin{equation}\label{eq:contract temp_lyap}
	\| \mathcal T \Psi - \mathcal T \Phi \|_{\lambda} \leq M  \| \Psi-  \Phi \|_{\lambda}, \quad  \Psi,\Phi \in \mathcal B_T.	
	\end{equation} 
	
\textit{Step 1:}  We first prove that $\Tcal (\mathcal B_T) \subset \mathcal B_T$. Fix $\Psi \in \mathcal B_{T}$ and  $t \leq T $. An application of the triangle inequality combined with the assumption \eqref{assumption:lyapunov_ODE} leads to 
\begin{align*}
\|(\mathcal T \Psi)_t \|_{L^1(\mu_1 \otimes \mu_2)} &\leq \kappa \int_{\R_+^2} |\mu_1|(d\theta) |\mu_2|(d\tau) \int_t^T e^{-(\theta + \tau)(s-t)}  ds \\
& \quad  +  \kappa \int_{\R_+^2} |\mu_1|(d\theta) |\mu_2|(d\tau) \int_t^T e^{-(\theta + \tau)(s-t)}
 \|\Psi_s\|_{L^1(\mu_1\otimes \mu_2)}  ds  \\
& \quad  +  \kappa \int_{\R_+^2} |\mu_1|(d\theta) |\mu_2|(d\tau) \int_t^T e^{-(\theta + \tau)(s-t)}   \int_{\R_+} |\mu_1|(d\theta') |\Psi_s(\theta',\tau)| \\
& \quad +  \kappa \int_{\R_+^2} |\mu_1|(d\theta) |\mu_2|(d\tau) \int_t^T e^{-(\theta + \tau)(s-t)} \int_{\R_+} |\Psi_s(\theta,\tau ')| |\mu_2|(d\tau '),\\
&=\kappa(\textbf{I}_t+\textbf{II}_t+\textbf{III}_t+\textbf{IV}_t).
\end{align*}
Recalling the definition \eqref{eq:Kbar}, an application of Tonelli's theorem and Cauchy--Schwarz inequality yields
 $$ \sup_{t\leq T}\textbf{I}_t  = \sup_{t \leq T} \int_t^T \bar K_1(s-t)  \bar K_2(s-t) ds  \leq   \|\bar K_1\|_{L^2(0,T)} \|\bar K_2\|_{L^2(0,T)} , $$
 which is finite  due to  \cite[Lemma  A.1]{abietal19a}. Similarly,  
\begin{align*}
\sup_{t\leq T}\textbf{II}_t \leq \| \Psi\|_{\mathcal B_T}  \|\bar K_1\|_{L^2(0,T)} \|\bar K_2\|_{L^2(0,T)}<\infty.
\end{align*} 
 Now, as $e^{-\tau(s-t)} \leq 1$, and $e^{-\theta(s-t)} \leq 1$,  for $s$ $\geq$ $t$, and $\theta,\tau$ $\in$ $\R_+$,  another application of Tonelli's theorem leads to 
\begin{align*}
\sup_{t\leq T}\textbf{III}_t \leq     \| \Psi\|_{\mathcal B_T}    \|\bar K_1\|_{L^1(0,T)}<\infty, & \quad \mbox{ and } 
\sup_{t\leq T}\textbf{IV}_t \leq     \| \Psi\|_{\mathcal B_T}    \|\bar K_2\|_{L^1(0,T)}<\infty.
\end{align*} 
Combining the above inequalities proves that $\|\Tcal \Psi\|_{\mathcal B_T}< \infty$ and hence $\Tcal : \mathcal B_{T} \to \mathcal B_{T}$. \\

\textit{Step 2:} We prove that there exists $\lambda>0$ such that~\eqref{eq:contract temp_lyap} holds.  Fix $\lambda>0$ and   $\Psi,\Phi \in \mathcal S_T$ such that $\|\Psi\|_{\lambda}$ and  $\|\Phi\|_{\lambda}$ are finite. Similarly to \textit{Step 1},  the triangle inequality and Tonelli's theorem  lead to 
	\bes{
		\sup_{t \leq T} e^{-\lambda (T-t)} \|(\mathcal T \Psi)_t - (\mathcal T \Phi)_t\|_{L^1(\mu_1 \otimes \mu_2)}   \leq  M(\lambda) \sup_{t \leq T} e^{-\lambda (T-t)} &\|\Psi_t- \Phi_t\|_{L^1(\mu_1 \otimes \mu_2)},
	}
	where 
	\bes{
		M(\lambda) &= \kappa  \int_0^T e^{-\lambda s} \left( \bar{K}_1(s) + \bar{K}_2(s) +  \bar{K}_1(s) \bar{K}_2(s)\right) ds.
	}
	By the dominated convergence theorem, $M(\lambda)$ tends to $0$ as $\lambda$ goes to $+\infty$.  We can therefore choose $\lambda_0 > 0$  so that  \eqref{eq:contract temp_lyap} holds with $M(\lambda_0) < 1$. {An application of the contraction mapping theorem  yields the existence and uniqueness statement in $(\mathcal B_T,\| \cdot\|_{\mathcal B_T})$ such that \eqref{def:infinite_dim_lyapunov1} holds, $\mu_1\otimes \mu_2-$a.e., for all $t \leq T$.   The interchange of the quantifiers is possible due to the  continuity of $t\mapsto \Psi_{t}(\theta,\tau)$ $\mu_1\otimes\mu_2$-a.e., which ends the proof.}
\end{proof}

\vspace{1mm}

\begin{lemma}\label{L:lyapunovestimate}
 The function $\Psi$ constructed in  Lemma~\ref{L:lyap_ST} satisfies the estimates \eqref{eq:estimatelyap2}-\eqref{eq:estimatelyap3}.
\end{lemma}

\begin{proof}
	We only prove \eqref{eq:estimatelyap3},  as \eqref{eq:estimatelyap2} follows by the same argument.
Integrating \eqref{def:infinite_dim_lyapunov1} over the $\tau$ variable leads to 
	\begin{align}\label{def:infinite_dim_lyapunov1int}
\!	\int_{\R_+} \!\! \Psi_t(\theta,\tau)\mu_2(d\tau)= 	\int_{\R_+}  \int_t^T e^{-(\theta + \tau)(s-t)} F(s,\Psi_s)(\theta,\tau)ds \mu_2(d\tau), \; { t\leq T, \; {\mu_1}-a.e.}
	\end{align}
	Let us define the $\mu_1$-null set
	$$ {\Ncal = \{\theta \in \R_+:   \eqref{def:infinite_dim_lyapunov1int} \mbox{ does not hold} \}}, $$
		and fix $\theta \in \R_+ \setminus \Ncal$ and $t\leq T$. The triangle inequality  on \eqref{def:infinite_dim_lyapunov1int} and  assumption \eqref{assumption:lyapunov_ODE} yields
	\bes{
		\int_{\R_+} |\mu_2|(d\tau) |\Psi_t(\theta,\tau)| & \leq \kappa \int_{\R_+} |\mu_2|(d\tau)  \int_t^T e^{-(\theta + \tau)(s-t)}  ds \\
				&\;\;\;\;+ \kappa \int_{\R_+} |\mu_2|(d\tau)\int_t^T e^{-(\theta + \tau)(s-t)}    \|\Psi_s\|_{L^1(\mu_1 \otimes \mu_2)}   ds\\
		&\;\;\;\;+ \kappa\int_{\R_+} |\mu_2|(d\tau) \int_t^T e^{-(\theta + \tau)(s-t)} \int_{\R_+} |\mu_1|(d\theta') |\Psi_s(\theta',\tau)| ds \\
		&\;\;\;\;+   \kappa\int_{\R_+} |\mu_2|(d\tau)\int_t^T e^{-(\theta + \tau)(s-t)} \int_{\R_+} |\mu_2|(d\tau ') |\Psi_s(\theta,\tau ')|  ds.
	}
	Using the  bound $e^{-\theta(s-t)} \leq 1$, an application of Tonelli's theorem gives  
\bes{
	\int_{\R_+} |\mu_2|(d\tau) |\Psi_t(\theta,\tau)|  \leq & \kappa \left(1 + \sup_{s \leq T} \|\Psi_s\|_{L^1(\mu_1 \otimes \mu_2)} \right) \int_0^T { \left(1+\bar{K}_2(s)\right)} ds\\
	&\quad + \kappa \int_t^T \Bar{K}_2(s-t) \int_{\R_+} |\mu_2|(d\tau ') |\Psi_s(\theta,\tau ')| ds. \label{eq:estimatetemplemma}
}
After a change of variable, we get that the function $f^\theta$ defined by 
$$f^\theta(t)=   	\int_{\R_+} |\mu_2|(d\tau) |\Psi_{T-t}(\theta,\tau)|, \quad  t \leq T,$$
satisfies the convolution inequality 
$$ f^\theta(t)  \leq c_T   + \kappa \int_0^t \bar K_2(t-s)  f^\theta(s) ds,$$
with $c_T = \kappa \left(1 + \sup_{s \leq T} \|\Psi_s\|_{L^1(\mu_1 \otimes \mu_2)} \right) \int_0^T { \left(1+\bar{K}_2(s)\right)} ds < \infty.$
It follows from \eqref{eq:estimatelyap1} that $f^\theta(t)$ is finite $\mu_1\otimes dt$--a.e., so that an application of the generalized Gronwall inequality for convolution equations, see \cite[Theorem 9.8.2]{GLS:90}, yields the estimate \eqref{eq:estimatelyap3}.
\end{proof}

\vspace{1mm}

\begin{lemma}
    \label{l:lyapunov_continuity}
    { 	Under the assumptions of Theorem~\ref{thm:existence_unicite_lyapunov}, let $t \in [0,T] \to \Psi_t$ be such that   \eqref{def:infinite_dim_lyapunov1}  holds, 
    with  \eqref{eq:estimatelyap1}, \eqref{eq:estimatelyap2} and \eqref{eq:estimatelyap3}.  Then,   $\Psi \in \mathcal C ([0,T],L^1(\mu_1\otimes \mu_2))$.}
\end{lemma}

\begin{proof}
	We first observe that by virtue of the boundedness of the coefficients \eqref{assumption:lyapunov_ODE} and the estimates \eqref{eq:estimatelyap1}, \eqref{eq:estimatelyap2} and \eqref{eq:estimatelyap3}, we have 
	  \begin{align}\label{eq:temp_boundF}
	  |F(s,\Psi_s)(\theta,\tau)| \leq  c, \quad dt \otimes \mu_1 \otimes \mu_2-a.e.
	  \end{align}
	for some constant $c$, where $F$ is given by \eqref{def:infinite_dim_lyapunov2}. 
Fix $s\leq t\leq T$. Using \eqref{def:infinite_dim_lyapunov1}, we write 
\begin{align*}
\Psi_s(\theta, \tau) - \Psi_t(\theta, \tau) &= \int_s^t e^{-(\theta + \tau)(u-s)} F(u,\Psi_u)(\theta,\tau)ds\\
&\quad +   \int_t^T \left( e^{-(\theta + \tau)(u-s)} - e^{-(\theta + \tau)(u-t)}\right) F(u,\Psi_u)(\theta,\tau)ds \\
&= \textbf{I}_{s,t}(\theta,\tau) + \textbf{II}_{s,t}(\theta,\tau)
\end{align*}
	$\mu_1 \otimes \mu_2$--a.e.  Integrating over the $\theta$ and $\tau$ variables and successive applications of Tonelli's theorem and Cauchy--Schwarz inequality  together with the bound \eqref{eq:temp_boundF} lead to  
	\begin{align*}
	\|\textbf{I}_{s,t}\|_{L^1(\mu_1\otimes \mu_2)} &\leq c \int_s^t \bar K_1(u-s) \bar K_2(u-s) du \\
	&\leq c  \| \bar K_1\|_{L^2(0,t-s)}  \| \bar K_2\|_{L^2(0,t-s)}.
	\end{align*}
 By virtue of the square integrability of $\bar K_1$ and $\bar K_2$, the right hand side goes to $0$ as $s \uparrow  t$. Similarly, using also that  $e^{-(\theta + \tau)(u-s)} \leq e^{-(\theta + \tau)(u-t)}$, we get 
 \begin{align*}
 \|\textbf{II}_{s,t}\|_{L^1(\mu_1\otimes \mu_2)} &\leq c \int_t^T \left( \bar K_1(u-t) \bar K_2(u-t) - \bar K_1(u-s)\bar K_2(u-s)\right) du\\
 &=   c \int_0^{T-t} \left( \bar K_1(u)\bar K_2(u) - \bar K_1(t-s+u)\bar K_2(t-s+u)\right) du\\
 &=   c \int_0^{T-t} \left( \bar K_1(u)-\bar K_1(t-s+u) \right)\bar K_2(u) du \\
 &\quad +      c \int_0^{T-t} \bar K_1(t-s+u)\left(  \bar K_2(u)- \bar K_2(t-s+u)\right) du\\
 &=c(\textbf{1}_{s,t}  + \textbf{2}_{s,t}).
 \end{align*}
 The right hand side goes also to $0$ as $s\uparrow t$. To see this, an application of Cauchy--Schwarz inequality gives 
 $$ \textbf{1}_{s,t}  \leq  \left( \int_0^T  \left( \bar K_1 (u)- \bar K_1(t-s+u) \right)^2 du \right)^{1/2} \| \bar K_2\|_{L^2(0,T)}.$$
 	Since $\bar K_1$ is an element of $L^2$, it follows from \cite[Lemma 4.3]{brezis2010functional} that 
 \begin{equation}\label{L2cont}
 \lim_{h \to 0}\int_0^T |\bar K_1(u+h)-\bar K_1(u)|^2 du = 0,
 \end{equation}
 showing that $\textbf{1}_{s,t} $ converges to $0$ as $s$ goes to $t$. Interchanging the roles of $\bar K_1$ and $\bar K_2$, we also get the convergence $\textbf{2}_{s,t} \to 0$ as $s\uparrow  t$. Combining the above leads to 
 $$ \| \Psi_s -\Psi_t\|_{L^1(\mu_1\otimes \mu_2)} \to 0, \quad \mbox{as } s \uparrow  t.$$
 Similarly, we get the same conclusion when $s\downarrow t$, and the proof is complete.
\end{proof}

\section{Solvability of the Riccati equation} \label{secRiccati}

The main goal of this section  is to prove Theorem \ref{T:Riccatiexistence}, i.e.,  the existence and uniqueness of a function $\Gamma$ $\in$  $C([0,T], L^1(\mu \otimes \mu))$ satisfying the  kernel Riccati equation \eqref{eq:Riccati_monotone_kernel_mild} (recall Definition \ref{D:nonnegative}),  and  the  estimate \eqref{eq:estimateRiccati}. 
This is obtained by  adapting the  technique  used  in classical linear-quadratic control  theory \citep[Theorem 6.7.2]{yong1999stochastic} to our setting with the following steps: 
 \begin{enumerate}
 \item  Construct a sequence of Lyapunov solutions $(\Gamma^i)_{i\geq 0}$  by successive iterations,
 \item  Establish the convergence of  $(\Gamma^i)_{i\geq 0}$  in $L^1(\mu\otimes \mu)$,
 \item  Prove that the limiting point is a solution to the Riccati equation \eqref{eq:Riccati_monotone_kernel_mild},
 \item  Derive  the  estimate \eqref{eq:estimateRiccati}, the continuity and the uniqueness for the Riccati solution.
 \end{enumerate}
 
  

\subsection{Step 1: Construction of a sequence of Lyapunov solutions $(\Gamma^i_t)_{i\geq 0}$}


\begin{lemma}
	\label{lemma:positive_lyapunov}
 Let $\Psi \in C([0,T],L^1(\mu \otimes \mu))$ denote the solution to the Lyapunov equation \eqref{def:infinite_dim_lyapunov1}  produced by Theorem \ref{thm:existence_unicite_lyapunov} for the configuration 
 \begin{equation}
    \label{assumption:symmetriclyap}
    \left\{
    \begin{array}{lc} 
    d_{11}=d_{21}=d,\quad d_{12}=d_{22}=d', & \quad  \mu_1 = \mu_2 = \mu,  \\ 
    \tilde B^1 = \tilde B^2 = \tilde B, &\quad  \quad  \tilde D^1 = \tilde D^2 = \tilde D.
    \end{array}
    \right.
\end{equation} 
and under the condition 
 \begin{equation}  \label{assumption:lyapunov_ODE_coeff} 
 \left\{
 \begin{array}{rcl}
    |\tilde Q_s(\theta,\tau)| & \leq &  \kappa, \quad dt\otimes \mu \otimes \mu-a.e. \\
    |\tilde B_s(\theta)| + |\tilde D_s(\theta)|  &\leq &  \kappa, \quad dt\otimes \mu-a.e. 
 \end{array}
 \right.
 \end{equation}
If {$\tilde Q_t \in \S^d_+(\mu \otimes \mu)$} for all $t \leq T$, then
\begin{enumerate}
	\item \label{L:symlyapi}
	  $t \mapsto \Psi_t $ is a non-increasing $\S^d_+(\mu \otimes \mu)$-valued function w.r.t  the order relation $\succeq_{\mu}$. 
	  \item  \label{L:symlyapii}
	  $t \mapsto \int_{\R_+^2}  \mu(d\theta)^\T \Psi_t(\theta, \tau) \mu(d\tau)$ is a non-increasing $\S^d_+$-valued function on $[0,T]$. 
\end{enumerate}
\end{lemma}
\begin{proof}Note that under \eqref{assumption:symmetriclyap}, the Lyapunov equation  \eqref{def:infinite_dim_lyapunov1} is invariant by transposition and exchange of $\theta$ and $\tau$. By uniqueness of the solution, we deduce that $\Psi_t(\theta, \tau) = \Psi_t(\tau,\theta)^\top$, $\mu \otimes \mu$--a.e., for all $t \leq T$. 
Fix $\varphi \in { L^{1}(\mu)}$ and $t \leq T$, and consider the following equation
\bec{
d \widetilde Y_s(\theta) &= \Big(- \theta\widetilde Y_s(\theta) + \int_{\R_+} \tilde B_s(\tau) \mu(d\tau) \widetilde  Y_s(\tau) \Big) ds 
+ \Big(\int_{\R_+} \tilde D_s(\tau) \mu(d\tau)  \widetilde Y_s(\tau) \Big)dW_s \\
\widetilde Y_t(\theta) &= \varphi(\theta),
}
which admits a unique  $L^1(\mu)$--valued solution such that 
	\begin{align}\label{eq:momentY}
	\sup_{t\leq s \leq T} \E\left[ \|\widetilde Y_s \|^4_{L^1(\mu)}  \right] < \infty,
	\end{align}
	see \cite[Theorem 4.1]{abietal19a}. 
Recall the bold notation  $\boldsymbol{G}$ in \eqref{defintegral}  for the integral operator generated by a  kernel $G$.  Note that, by virtue of Remark~\ref{R:Linfty},
		  $\Psi_s \in L^{\infty}(\mu\otimes \mu)$, so that $\boldsymbol{\Psi}_s:L^1(\mu)\to L^{\infty}(\mu^\top)$, for all $s \leq T$. This shows that  $s\mapsto \langle \widetilde Y_s, \boldsymbol{ \Psi_s} \widetilde Y_s \rangle_{\mu}$ is well-defined $\P$-a.s.  An application of  
It\^o's formula  (see \cite[Lemma 5.1]{abietal19a}) to the process $s\to \langle \widetilde Y_s, \boldsymbol{ \Psi_s} \widetilde Y_s \rangle_{\mu}$
yields, due to the vanishing terminal condition for $\Psi$, and  after successive applications of Fubini's theorem: 
		\bes{
		0&=  \langle \varphi, \boldsymbol{ \Psi_t} \varphi \rangle_{\mu} -  \int_t^T \langle \widetilde Y_s, \boldsymbol{\tilde Q}_s \widetilde Y_s \rangle_{\mu} ds \\
		&\quad +  \int_t^T  \int_{\R_+}  \widetilde Y_s(\theta)^\top \mu(d\theta)^\top  \tilde D_s(\theta)^\top  \int_{\R_+^2}\mu(d\theta')^\top \Psi_s(\theta',\tau') \mu(d\tau ') 
		\widetilde Y_s(\tau ') dW_s \\
		&\quad +  \int_t^T  \int_{\R_+^2}  \widetilde Y_s(\theta ')^\top \mu(d\theta ')^\top \Psi_s(\theta ',\tau ') \mu(d\tau ') \int_{\R_+}   \tilde D_s(\tau)  \mu(d\tau) \widetilde  Y_s(\tau)  dW_s. \label{eq:postivityeq}
	}
It  is straightforward to check that the local martingales terms are in fact true martingales due to the boundedness conditions \eqref{assumption:lyapunov_ODE_coeff}, \eqref{eq:estimatelyap2} and  the moment bound \eqref{eq:momentY}. Thus,  taking the expectation on both sides of \eqref{eq:postivityeq} yields that
	\bes{
		 \langle \varphi, \boldsymbol{\Psi_t} \varphi \rangle_{\mu} = \E \left[ \int_t^T \langle \widetilde Y_s,  \boldsymbol{\tilde Q}_s \widetilde Y_s \rangle_{\mu} ds\right],
    }
    which  ensures the positiveness and the non increasingness of $t  \mapsto \langle \varphi , \boldsymbol{\Psi}_t \varphi \rangle_{\mu}$ for any $\varphi \in L^1(\mu)$, since 
    $s\to \tilde Q_s $ is $\mathbb S^{d}_+(\mu \otimes \mu)$-valued. {This proves Assertion~\ref{L:symlyapi}.}  
Next, by considering the sequence of $L^1(\mu)$-valued functions {$(\varphi^n(\theta)=z \ind_{[1/n,\infty)}(\theta))_{n\geq 1}$ for arbitrary $z \in \R^{d'}\setminus\{0\}$}, 
using that $\Psi_t \in L^1(\mu \otimes \mu)$, and taking the limit, we obtain that $t \mapsto \int_{\R_+^2} \mu(d\theta)^\T \Psi_t(\theta, \tau) \mu(d\tau)$ is a non increasing $\mathbb S_+^{d}$-valued function. This proves Assertion~\ref{L:symlyapii} and concludes the proof. 
\end{proof}

\vspace{1mm}

From now on, we work under assumption \eqref{assumption:QN}.  We construct   a sequence of Lyapunov solutions $(\Gamma^i)_{i\geq 0}$ by induction as follows.

 \vspace{1mm}

\noindent $\bullet$ \textit{Initialization:} Let $\Gamma^0\in C([0,T],L^1(\mu\otimes \mu))$ be the unique solution given by Theorem~\ref{thm:existence_unicite_lyapunov} 
 to the following Lyapunov equation
		\bec{
	\Gamma^0_t(\theta, \tau) &=  \int_t^T e^{-(\theta + \tau)(s-t)} F_0(\Gamma^0_s)(\theta,\tau)ds,
 \\
	F_0(\Gamma)(\theta, \tau)&=Q+ D^\T \int_{\R_+^2} { \mu(d\theta')^\T \Gamma(\theta',\tau ') \mu(d\tau ')} D \\
	&\quad  + B^\top \int_{\R_+} \mu(d\theta')^\T \Gamma(\theta',\tau)  + \int_{\R_+} \Gamma(\theta,\tau ') \mu(d\tau ') B.
		\label{eq:lyapunov_0}
}
Since $Q\in \S^d_+$, an application of Lemma  \ref{lemma:positive_lyapunov}-\ref{L:symlyapii} yields that 
$F^\top \int_{\R_+^2} \mu(d\theta')^\top \Gamma^0_t(\theta',\tau ') \mu(d\tau ') F \in \S^m_+$, for all $t\leq T$. 
Combined with the assumption $N - \lambda I_m \in \S^m_+$, we obtain 
\bes{\label{eq:condinvN}
N + F^\top \int_{\R_+^2} \mu(d\theta')^\top \Gamma^0_t(\theta',\tau ') \mu(d\tau ')   F - \lambda I_m \in \S^m_+,\quad  t \leq T.
}
	
\vspace{1mm}	
	
\noindent $\bullet$ \textit{Induction}: for  $i \in \N$, having constructed  $\Gamma^i \in  C([0,T],L^1(\mu\otimes \mu))$  such that 
\begin{align}\label{eq:gammaisd}
N + F^\top \int_{\R_+^2} \mu(d\theta')^\top \Gamma^i_t(\theta',\tau ') \mu(d\tau ')   F - \lambda I_m \in \S^m_+,\quad  t \leq T.
\end{align}
and 
\begin{align}\label{eq:tempestimatelyap3}
	\int_{\R_+} |\mu|(d\tau)|\Gamma^i_t(\theta,\tau)| \leq \kappa_i',\quad \mu-a.e., \quad t\leq T,
\end{align}
for some $\kappa_i'>0$, we  define 
\begin{align}
\Theta^i_t(\tau) =& -\left(N + F^\T \int_{\R_+^2} \mu(d\theta')^\T \Gamma^i_t(\theta',\tau ') \mu(d\tau ')   F\right)^{-1} \label{deftheta} \\
& \quad\quad   \quad \quad  \times \left(F^\T \int_{\R_+^2} \mu(d\theta')^\T \Gamma^i_t(\theta',\tau ') \mu(d\tau ')  D +C^\T \int_{\R_+} \mu(d\theta') \Gamma^i_t(\theta',\tau) \right ), \nonumber
\end{align}
 together with the coefficients 
\begin{align}
    \label{coeff_recu_lyapu}
    \tilde Q^i_t(\theta,\tau)=Q+ \Theta^i_t(\theta)^\T N \Theta^i_t(\tau), \quad \tilde B^i_t(\tau) = B + C \Theta^i_t(\tau), \quad \tilde D^i_t(\tau) = D + F \Theta^i_t(\tau).
\end{align}
	Since $\Gamma^i \in C([0,T],L^1(\mu\otimes \mu))$, we have 
	$$ \sup_{t\leq T} \|\Gamma^{i}_t\|_{L^1(\mu \otimes \mu)}< \infty.$$
	Combined with the estimate \eqref{eq:tempestimatelyap3}, this yields the existence of $c_i>0$ such that 
	$$  \Theta^i_t(\theta) \leq c_i, \quad \mu-a.e., \quad t \leq T.$$
This implies that the coefficients {$\tilde Q^i,\tilde B^i,\tilde D^i$} satisfy \eqref{assumption:lyapunov_ODE_coeff}. Therefore,  Theorem~\ref{thm:existence_unicite_lyapunov} can be applied to get the existence of a unique solution $\Gamma^{i+1}\in C([0,T],L^1(\mu\otimes \mu))$   to the following Lyapunov equation
	\bec{
	    \label{eq:recu_gamma_i}
    	\Gamma^{i+1}_t(\theta, \tau) &=  \int_t^T e^{-(\theta + \tau)(s-t)} F_i(s,\Gamma^{i+1}_s)(\theta,\tau)ds,
    	\\
    	F_i(s,\Gamma)(\theta, \tau)&= \tilde Q^i_s(\theta,\tau)+ \tilde D_t^i(\theta)^\T \int_{\R_+^2} { \mu(d\theta')^\T \Gamma(\theta',\tau ') \mu(d\tau ')} \tilde D_t^i(\tau) \\
    	&\quad  + \tilde B_t^i(\theta)^\top \int_{\R_+} \mu(d\theta')^\T \Gamma(\theta',\tau)  + \int_{\R_+} \Gamma(\theta,\tau ') \mu(d\tau ') \tilde B_t^i(\tau),
    }
such that the estimate \eqref{eq:tempestimatelyap3} holds also for $\Gamma^{i+1}$. Furthermore, since $\tilde Q^i_t$  clearly lies in $\S^d_{+}(\mu\otimes \mu)$, for all $t\leq T$,   Lemma \ref{lemma:positive_lyapunov}-\ref{L:symlyapii} yields that \eqref{eq:gammaisd} is satisfied with $\Gamma^i_t$  replaced by $\Gamma^{i+1}_t$, for all $t\leq T$. 
This ensures that the induction is well-defined. 

\subsection{Step 2: Convergence of $(\Gamma^i_t)_{i\geq 0}$ in $L^1(\mu\otimes \mu)$}

\begin{lemma}\label{lemma:simple_convergence_lyapu} For $i\in \N$ and for a  scalar function $\xi \in L^{\infty}(|\mu|)$  define the matrix-valued functions 
	$$\Uc^i = \int_{\R_+^2} \mu(d\theta)^\T \Gamma^i(\theta,\tau)\mu(d\tau), \quad \Vc^i(\xi) = \int_{\R_+^2} \mu(d\theta)^\T \Gamma^i(\theta,\tau)\mu(d\tau) \xi(\tau).$$
Then,
	\begin{enumerate} 
		\item \label{Lconvi} $\left(\Uc^i\right)_{i\geq 0}$  is a non-increasing sequence (meaning $\Uc^{i+1}_t \leq \Uc^i_t$,  $i\in \N$ and $t\leq T$) of monotone non-increasing functions on the space $C([0,T],\S^d_+)$,  converging pointwise to a limit denoted by $\Uc$; 
		\item  \label{Lconvii} $\left(\Vc^i(\xi)\right)_{i\geq 0}$ is a uniformly bounded sequence of functions on the space  $C([0,T],{\R^{d'\times d'}})$,  
		converging pointwise to a limit denoted by $\Vc(\xi)$, for any scalar function $\xi \in L^{\infty}(|\mu|)$.
	\end{enumerate}
\end{lemma}

\begin{proof}  Throughout the proof  we consider the intermediate scalar sequences
	\begin{align}\label{eq:Utif}
	{U_t^i(\varphi) = \langle \boldsymbol{\Gamma}^i_t \varphi , \varphi \rangle_{\mu^\top} }
	\quad \mbox{and}  \quad    V_t^i(\varphi, \psi) =   \langle  \boldsymbol{\Gamma}^i_t \varphi , \psi \rangle_{\mu^\top},  \quad t \leq T,  \quad i \in \N, \quad \varphi,\psi \in L^{\infty}(\mu),
	\end{align}
	which are well-defined since $\Gamma_t^i \in L^{1}(\mu\otimes \mu)$ for all $t\leq T$. We also Set $\Theta^{-1}=0$, and for  $i \geq 0$, and define
	\bes{
		\Delta^i = \Gamma^{i} - \Gamma^{i+1}, \;\;\;\;\;\; \rho^i = \Theta^{i-1} - \Theta^{i},
	}
	where we recall that $\Theta^i$ is given by \eqref{deftheta}.  	
	Straightforward computations, detailed in Lemma \ref{l:eq_gamma_i_recu} and Remark~\ref{R:deltai}, yield that $\Delta^i$ solves the Lyapunov equation:
	\bec{
		\Delta^{i}_t(\theta, \tau) &=  \; \int_t^T e^{-(\theta + \tau)(s-t)} F^{\delta}_i(s,\Delta^{i}_s)(\theta,\tau)ds,
		\\
		F^{\delta}_i(t,\Delta)(\theta, \tau)&= \; \tilde Q^{i,\delta}_t{(\theta,\tau)}+ \tilde D_t^i(\theta)^\T \int_{\R_+^2} { \mu(d\theta')^\T \Delta(\theta',\tau ') \mu(d\tau ')} \tilde D_t^i(\tau) \\
		&\quad  + \tilde B_t^i(\theta)^\top \int_{\R_+} \mu(d\theta')^\T \Delta(\theta',\tau)  + \int_{\R_+} \Delta(\theta,\tau ') \mu(d\tau ') \tilde B_t^i(\tau), \label{eq:Deltai}
	}
	where 
	$$ \tilde Q^{i,\delta}_t (\theta,\tau)=\rho^i_t(\theta)^\T \left(N + F^\T \int_{\R_+^2} \mu(d\theta')^\T \Gamma^i_t(\theta',\tau ')\mu(d\tau ') F\right)\rho^i_t(\tau).$$
\noindent  $\bullet$  
Fix $i\in \N$. Since $ \tilde Q^{i,\delta}_t \in \S^d_+(\mu \otimes \mu)$,  an application of Lemma  \ref{lemma:positive_lyapunov}-\ref{L:symlyapi}   on \eqref{eq:Deltai} shows that  $t \mapsto \Delta^i_t $ is a non-increasing $\S^d_+(\mu \otimes \mu)$-valued function. Thus,  for any  $\varphi\in L^{1}(\mu)$, 
	\bes{
	\langle  \varphi, \boldsymbol{\Gamma}^0_0 \varphi\rangle_{\mu} \geq 	\langle  \varphi, \boldsymbol{\Gamma}^0_t \varphi\rangle_{\mu}  \geq \langle  \varphi, \boldsymbol{\Gamma}^i_t \varphi\rangle \geq \langle  \varphi, \boldsymbol{\Gamma}^{i+1}_t \varphi\rangle_{\mu} \geq  0, \quad  t \leq T, \quad  i \in \N,
		\label{eq:decreasing_seq}
}
Since for all $t\leq T$, $\Gamma_t^i$ is also an element of  $L^{1}(\mu\otimes \mu)$,  the density of simple functions in $L^\infty(\mu)$ with respect to the uniform norm,  implies that 
		\bes{\label{eq:Udecrease}
			0  &  \leq U^{i+1}_t(\varphi)  \leq U^i_t(\varphi) \leq U^0_t(\varphi) \leq U^0_0(\varphi).
}
for all $\varphi \in L^{\infty}(\mu)$.  This implies that the sequence of functions $( U^i(\varphi))_{i\geq 0}$ is non-increasing, nonnegative and converging pointwise to a limit that we denote by 
	$U_t(\varphi)$ for any $t\in [0,T]$. Furthermore,  $t\mapsto U_t^i(\varphi)$ is continuous, for all  $i$ $\in$ $\N$ and    $\varphi \in L^\infty(\mu)$, thanks to  the continuity of 
	$t\to \Gamma_t$  in $L^1(\mu \otimes \mu)$, see Lemma \ref{l:lyapunov_continuity}. The claimed statement \ref{Lconvi} for $\mathcal U$ now follows by evaluating with $\varphi(\theta)\equiv z$, where $z$ ranges through $\R^{d'}$.	
	
\noindent $\bullet$  Since $\Gamma^i_t - \Gamma^j_t \in \S^d_+(\mu\otimes \mu)$ for any $i \leq j$, an application of  the Cauchy-Schwartz inequality (see Lemma \ref{l:CS}) yields 
	\bes{\label{eq:CS1}
		\langle \varphi, \left( \bold{\Gamma}^i_t -\bold{\Gamma}^j_t \right) \psi \rangle_{\mu}^2   &\leq  \langle \varphi, \left( \bold{\Gamma}^i_t -\bold{\Gamma}^j_t \right) \varphi \rangle_{\mu} \, \langle \psi, \left( \bold{\Gamma}^i_t -\bold{\Gamma}^j_t \right) \psi \rangle_{\mu}, \quad \varphi,\psi \in L^1(\mu).}
	Invoking once again   the density of simple functions in $L^\infty(\mu)$ with respect to the uniform norm and the fact that for all $t\leq T$, $\Gamma_t \in L^{1}(\mu\otimes \mu)$, \eqref{eq:CS1}  gives
		\bes{
	\left( V^i_t(\varphi,\psi) -V^j_t(\varphi,\psi) \right)^2	\leq  \left(U^i_t(\varphi)- U^j_t(\varphi) \right)\left(U^i_t(\psi)- U^j_t(\psi)\right), \quad \varphi,\psi \in L^{\infty}(\mu).}
Whence, the sequence of real valued functions $\left(t \mapsto V_t^i(\varphi,\psi)\right)_{i \geq 0}$ is uniformly bounded. Furthermore, this also shows that  the sequence 
	$\left(t \mapsto V_t^i(\varphi,\psi)\right)_{i \geq 0}$  is a real-valued  Cauchy sequence that converges pointwise to a limit that we denote by $V_t(\varphi,\psi)$, for any $\varphi,\psi \in L^{\infty}(\mu)$. 
	To obtain the continuity of $V^i(\varphi,\psi)$,  note that $\Gamma^i_t - \Gamma^i_s \in \S^d_+(\mu\otimes \mu)$ for any $t \leq s$ and $\Gamma^i_s - \Gamma^i_t \in \S^d_+$ for any $s \leq t$, 
	which allows us once again to apply the Cauchy Schwartz inequality (see Lemma \ref{l:CS}) coupled with the density argument to obtain for any $s,t \in [0,T]$: 
	\bes{
		\left( V^i_t(\varphi,\psi) -V^i_s(\varphi,\psi) \right)^2  \leq & \left(U^i_t(\varphi)- U^i_s(\varphi) \right)\left(U^i_t(\psi)- U^i_s(\psi)\right).
	}
	Consequently, the continuity of $U^i(\varphi)$ for any $\varphi \in L^\infty(\mu)$ implies that  of $V^i(\varphi,\psi)$ for any $\varphi,\psi \in L^\infty(\mu)$.   Fix $\xi \in L^{\infty}(|\mu|)$, the claimed statement \ref{Lconvii} for $\mathcal V(\xi)$ now follows by evaluating with $\varphi(\theta)\equiv z$ and $\psi(\theta)=\xi(\theta) z'$, where $z,z'$ range through $\R^{d'}$.
\end{proof}

\vspace{1mm}

\begin{lemma}
\label{lemma:simple_int_uniform_bound}
There exists a constant $\kappa>0$ such that for every $i \in \N$ and $t\in [0,T]$
    \bes{
        \label{eq:rec_hyp}
         \left| \int_{\R_+} \Gamma^{i}_t(\theta, \tau) \mu(d\tau) \right| \leq \kappa \quad  \mu-a.e.
    }
\end{lemma}

\begin{proof}
Lemma \ref{lemma:simple_convergence_lyapu}  ensures that there  exists  a constant $M>0$ such that for every $i \in \N$
\bes{
	\label{uniform_bound_double_int}
    \sup_{t \leq T}\left|\int_{\R_+^2}\mu(d\theta)^\T \Gamma^i_t(\theta,\tau)\mu(d\tau)\right| \vee \sup_{t \leq T} \left|\int_{\R_+^2}e^{-\theta t }\mu(d\theta)^\T \Gamma^i_t(\theta,\tau)\mu(d\tau)\right| \leq M.
}
Fix $i \in \N$. We proceed as in the proof of Lemma~\ref{L:lyapunovestimate} to bound the 
quantity $g^i_t(\theta)= \left| \int_{\R_+} \Gamma^{i+1}_t(\theta, \tau) \mu(d\tau) \right|$. By construction $\Gamma^{i+1}$ solves \eqref{eq:recu_gamma_i}, so that an  integration over the $\tau$-variable combined with \eqref{uniform_bound_double_int} and the triangle inequality yield 
\bes{
    \label{ineq:Gamma_iGamma_iplus1}
   g^{i+1}_t(\theta) \leq & 4 r T +   4 r \int_t^T (1+ |K(t-s)|)\left(   g^i_s(\theta) +    g^{i+1}_s(\theta)\right) ds, 
}
where $r$ is a constant only depending on $B,D,Q$ $N$ and $M$.  
Let us now show the desired inequality \eqref{eq:rec_hyp}. For $n \geq 0$, let us  define 
\bes{
    G^n_t(\theta) = \sup_{i=0,\ldots,n}  \left| \int_{\R_+} \Gamma^{i}_t(\theta,\tau)\mu(d\tau) \right|.
    } 
The inequality \eqref{ineq:Gamma_iGamma_iplus1} yields for every $i\geq 0$
\bes{
    G^{i+1}_t(\theta)  \leq 4 r T +   4 r \int_t^T  \left(1+ 2|K(t-s)|\right) G^{i+1}_s(\theta) ds.
    }
Consequently, the generalized Gronwall inequality implies that there exists a constant c only depending on $B,C,D,F,N,T,K$ and $M$ such that for every $n\in\N$, $t \in [0,T]$ we have $|G_t^n(\theta)|\leq c$ for $\mu$-almost  every $\theta$ and $t\in [0,T]$.
\end{proof}

\vspace{1mm}

\begin{lemma}
\label{L:uniform_convergence_U}
The sequence of functions $(\Uc^i)_{i \geq 0}$ 
converges uniformly on $C([0,T],\S^d_+)$ to its simple limit $\Uc$ introduced in Lemma \ref{lemma:simple_convergence_lyapu}. 
\end{lemma}
\begin{proof}
From Lemma \ref{lemma:simple_convergence_lyapu},  we have that $(\Uc^i)_{i \geq 0}$ is a non increasing sequence of continuous functions converging pointwise to  $\Uc$. To obtain the uniform convergence it suffices to show that $\Uc$ is continuous and apply Dini's theorem.  To do so our strategy is to show that $t \mapsto \Uc_t$ solves an equation whose solutions are continuous. 

\textit{Step 1. Equation satisfied by $\Uc$.}
By definition $\Gamma^{i+1}$ is solution to \eqref{eq:recu_gamma_i}, thus by integrating over $\tau,\theta$ and applying Fubini's theorem we get
\bec{
	\label{eq:U_z}
    	\Uc^{i+1}_t  &=   \int_t^T  \tilde{F}_i(t,r) \left(\Uc^{i+1} \right) dr,
    	\\
    	\tilde{F}_i(t,r) (\Uc^{i+1})    &=  \textbf{I}^i(t,r)  +  \textbf{II}^i(t,r)   + \textbf{III}^i(t,r) +\textbf{III}^i(t,r)^\T ,        \\
}
where 
\begin{align}
 \textbf{I}^i(t,r)  &= \int_{\R_+^2}  e^{-\theta(r-t)}  \mu(d\theta)^\T \tilde Q^i_r(\theta, \tau) \mu(d\tau) e^{-\tau(r-t)}, \\
  \textbf{II}^i(t,r)  &= \left( \int_{\R_+} e^{-\theta(r-t)} \mu(d\theta)^\T \tilde D_r^i(\theta)^\T \right)  \Uc_r^{i+1}  \left(  \int_{\R_+} \tilde D_r^i(\tau) \mu(d\tau)e^{-\tau(r-t)} \right)  , \\
  \textbf{III}^i(t,r)  &= \left(  \int_{\R_+}  e^{-\theta(r-t)}\mu(d\theta)^\T   \tilde B_r^i(\theta)^\top \right)   \Vc_r^i (e^{\cdot(t-r)})  ,   
\end{align}    
with $\tilde B^i,\tilde D^i$ and $\tilde Q^i$ defined as in \eqref{coeff_recu_lyapu}. The pointwise convergences and the uniform bounds stated in Lemma  \ref{lemma:simple_convergence_lyapu} allow us to apply the dominated convergence  theorem to \eqref{eq:U_z} to get
\bec{
	\label{eq:limit_U}
	\Uc_t  &= \int_t^T  \tilde{F}(t,r) \left(\Uc \right)  dr, \\
	\tilde{F}(t,r)  \left(\Uc  \right)   &=  \textbf{I}(t,r)  +  \textbf{II}(t,r)   + \textbf{III}(t,r) +\textbf{III}(t,r)^\T ,    \\
}
where
\begin{align}
\textbf{I}(t,r)   &=  K(r-t)^\T Q  K(r-t) +  \tilde{\Theta}(t,r)^\T N   \tilde{\Theta}(t,r), \\
\textbf{II}(t,r)  &= \left(DK(r-t) + F^\T   \tilde{\Theta}(t,r) \right)^\T \Uc_r  \left(DK(r-t) + F^\T   \tilde{\Theta}(t,r) \right) ,\\
\textbf{III}(t,r) &= \Vc_r\big(e^{\cdot(t-r)}\big) \left(B K(t-r) + C^\T \tilde{\Theta}(t,r)\right) , \\
\tilde{\Theta}(t,r) &= -\left(N + F^\T \Uc_r F\right)^{-1}   \left(F^\T \Uc_r D K(r-t) + C^\T \Vc_r\big(e^{\cdot(t-r)}\big)   \right).
\end{align}    

\textit{Step 2. Continuity of $t \mapsto \Uc_t$.}  We first observe that by virtue of Lemma \ref{lemma:simple_convergence_lyapu} $t\mapsto \mathcal U_t$ and $t\mapsto \mathcal V_t$ are bounded on $[0,T]$ so that  there exists $c>0$ such that
\bes{
	\label{bound:F_tilde}
	\left |\tilde{F}(t,r)(\Uc) \right| \leq c \left(1 + |K(r-t)|^2 \right), \quad  t \leq r \leq T.
}
Fixing  $ t \leq s \leq T$, it follows from \eqref{eq:limit_U} that 
\begin{align*}
\Uc_t  - \Uc_s  &= \;   \int_t^s  \tilde{F}(t,r) \left( \Uc  \right)  dr  + \int_s^T \left( \tilde{F}(t,r) \left(\Uc \right)  -  \tilde{F}(s,r) \left(\Uc \right)\right)  dr \\
		      & = \;  \textbf{1}_{t,s} + \textbf{2}_{t,s}.
\end{align*}
By \eqref{bound:F_tilde}, 
\bes{
	\left| \textbf{1}_{t,s} \right| \leq c \left( s-t + \|K\|^2_{L^2(0,s-t)} \right).
}
By virtue of the square integrability of $K$, the right hand side goes to 0 as $t \uparrow s$. Similarly, using $u^\T Q u - v^\T Q v = (u+v)^\T Q (u-v) $, we get
\bes{
	\left | \textbf{2}_{t,s} \right| \leq& c \int_s^T \left| K(r-t) - K(r-s) \right | \left| K(r-t) + K(r-s)  \right | dr \\
	&+c \int_s^T |\Vc_r(e^{\cdot(t-r)}) -  \Vc_r(e^{\cdot(s-r)}) | dr \\
	\leq & \;  \textbf{A}_{t,s} + \textbf{B}_{t,s}, 
}
where $c$ is a constant. The first term can be easily handled with Cauchy-Schwarz inequality
\bes{
  \textbf{A}_{t,s}  \leq & 2 c \|K\|^2_{L^2(0,T)} \int_0^T \left| K(r + s-t) - K(r) \right |^2 dr, 
}
 which shows that $\textbf{A}_{t,s} $ converges to zero as $t$ goes to $s$, recall \eqref{L2cont}. For the second term note that for all  $i \in \N$, $t \leq s\leq r\leq T$, 
 \bes{
 	\label{bound_diff_V_i}
 	\Bigg | \int_{\R_+^2} & \mu(d\theta)^\T \Gamma^i_r(\theta, \tau)   \mu(d\tau) e^{-\tau(r-t)} - \int_{\R_+^2} \mu(d\theta)^\T \Gamma^i_r(\theta, \tau)   \mu(d\tau) e^{-\tau(r-s)} \Bigg | \\
	& \leq   \esssup_{\tau' \in \R_+} \left| \int_{\R_+} \mu(d\theta)^\T \Gamma^i_r(\theta, \tau')   \right|  \int_{\R_+}  e^{-\tau(r-t)} - e^{-\tau(r-s)}   |\mu|(d\tau) \\
	&\leq  \kappa \left |\bar K (r-s) - \bar K(r-t) \right |,
 }
 where $\kappa$ is the uniform bound from Lemma \ref{lemma:simple_int_uniform_bound} and 	\begin{align}\label{eq:kbarproof}
 \bar K(t)= \int_{\R_+} e^{-\theta t } |\mu|(d\theta), \quad t>0.
 \end{align}
Taking the limit  $i$ $\rightarrow$ $\infty$  in \eqref{bound_diff_V_i}  and invoking Lemma \ref{lemma:simple_convergence_lyapu}, we obtain 
 \bes{
 	\label{bound_diff_V}
	\left| \Vc_r(e^{\cdot(t-r)}) -  \Vc_r(e^{\cdot(s-r)})  \right | \leq \kappa \left |\bar K (r-s) - \bar K(r-t) \right |.
 }
 Thus, similarly as for $\textbf{A}_{t,s}$ we get that  $\textbf{B}_{t,s}$ converges to $0$ as $t$ goes to $s$. As a result $\Uc$ is continuous.
\end{proof}

\vspace{1mm}

\begin{lemma}\label{L:cauchy}
    For any $t\leq T$,  $(\Gamma_t^i)_{i\geq 0}$ is a Cauchy sequence in $L^1(\mu\otimes \mu)$.
\end{lemma}

\begin{proof}
Let $t\leq T$ and $i\leq j$. Let $\Theta^j, \tilde B^j, \tilde D^j$ be defined as in \eqref{coeff_recu_lyapu} for any $j\in\N$. Then cumbersome but straightforward computations, detailed in Lemma \ref{l:eq_gamma_i_recu}, yield that  $\Delta_t^{i,j}=\Gamma^i_{t}-\Gamma^{j}_t$  solves the Lyapunov equation 
\bec{
	\Delta^{ij}_t(\theta, \tau) &= \int_t^T e^{-(\theta + \tau)(s-t)} F^{\delta}_{ij}(s,\Delta^{ij}_s)(\theta,\tau)ds,
	\\
	F^{\delta}_{ij}(t,\Delta)(\theta, \tau)&= \tilde Q^{ij,\delta}_t{ (\theta,\tau)}+ \tilde D_t^{j-1}(\theta)^\T \int_{\R_+^2} { \mu(d\theta')^\T \Delta(\theta',\tau ') \mu(d\tau ')} \tilde D_t^{j-1}(\tau) \\
	&\quad  + \tilde B_t^{j-1}(\theta)^\top \int_{\R_+} \mu(d\theta')^\T \Delta(\theta',\tau)  + \int_{\R_+} \Delta(\theta,\tau ') \mu(d\tau ') \tilde B_t^{j-1}(\tau) \\
	& \quad + S^{ij}_t(\theta)^\T \rho^{ij}_t(\tau) +\rho^{ij}_t(\theta)^\T  S^{ij}_t(\tau),   
	 \label{eq:Delta_ijCauchy}
}
where
\begin{align*}  
\rho^{ij}&=\Theta^{i-1} -\Theta^{j-1}, \\
\tilde Q^{ij,\delta}_t (\theta, \tau)&=\rho^{ij}_t(\theta)^\T \left(N + F^\T  \mathcal U^i_t F\right)\rho^{ij}_t(\tau),\\   
S^{ij}_t(\tau) &= C^\T \int_{\R_+} \mu(d\theta')^\T  \Gamma^i_t(\theta',\tau) + F^\T \mathcal U^i_t D  + \Big(N + F^\T  \mathcal U^i_t F \Big)\Theta^{j-1}_s(\tau),
\end{align*}
and $\mathcal U$ is defined as in Lemma~\ref{lemma:simple_convergence_lyapu}.
We will show that  $\|\Delta^{ij}_t\|_{L^1(\mu \otimes \mu)}\to 0$ as $i,j\to \infty$ by successive applications of Gronwall inequality and by showing that $\rho^{ij}$ is small enough. For this, we fix $\epsilon > 0$ and  we denote by $c>0$ a scalar independent of $i,j,t, \tau$  and $\theta$ that may vary from line to line throughout the proof. 

\textit{Step 1. We bound the terms $|\rho^{ij}_t(\tau)|$ and   $\left|\int_{\R_+} \Delta^{ij}_t(\theta,\tau)\mu(d\tau) \right|$.} We first write 
\bes{
	\label{eq:rho_ij}
	\rho^{ij}_t(\tau) =&\left(  \left(N + F^\T \mathcal U^j_t  F\right)^{-1} - \left(N + F^\T \mathcal U^i_t   F\right)^{-1} \right)  \left( F^\T \mathcal U^j_t D +C^\T \int_{\R_+} \mu(d\theta')^\T \Gamma^j_t(\theta',\tau)  \right)\\
	&-\left(N + F^\T \mathcal U^i_t   F\right)^{-1}   \left(F^\T \left(\mathcal U^i_t-\mathcal U^j_t\right)  D +C^\T  \int_{\R_+} \mu(d\theta')^\T \Delta^{ij}_t(\theta',\tau)  \right).  
}
 By the  uniform convergence of  the sequence of functions $\left(\mathcal U^i\right)_{i \geq 0}$, obtained in  Lemma \ref{L:uniform_convergence_U},  one can find $n' \in \N$  such that 
\bes{
	\label{bound:uniform_delta_ij}
	&\left|\mathcal U^i_t-\mathcal U^j_t  \right| +  \left|  \left(N + F^\T \mathcal U^j_t F\right)^{-1} - \left(N + F^\T \mathcal U^i_t   F\right)^{-1} \right|   \leq \epsilon, \quad t\leq T, \quad i,j\geq n',
}
where the bound for the second term comes from  the matrix identity $A^{-1}-B^{-1}=B^{-1}(B-A)A^{-1}$.
Furthermore, it follows from Lemmas \ref{lemma:simple_convergence_lyapu} and \ref{lemma:simple_int_uniform_bound} that 
\bes{
    \label{bound:int_intint}
    \left|\mathcal U^i_t  \right| \vee \left|\int_{\R_+} \Gamma^i_t(\theta,\tau)\mu(d\tau) \right| \leq c, \quad \mu-a.e., \quad t\leq T, \quad i \geq 0.
}
   Combining the previous identity with  \eqref{bound:uniform_delta_ij} and \eqref{eq:rho_ij} yields 
\begin{align}\label{eq:rhobound}
|\rho^{ij}_t(\tau)|\leq   c\left(\epsilon +  \left|\int_{\R_+} \Delta^{ij}_s(\theta,\tau)\mu(d\tau)\right|\right), \quad \mu-a.e., \quad t\leq T, \quad i,j\geq n'.
\end{align}
In addition, \eqref{bound:int_intint} yields that 
\begin{align}
\label{bound:Theta}
|\Theta^i_t(\theta)| \leq  c, \quad \mu-a.e., \quad t\leq T, \quad i \geq 0,
\end{align}
which in turn implies 
\begin{align}
    \label{bound:Theta2}
\int_{\R_+}e^{-\theta(s-t)}\left|\Theta^i_t(\theta)\right||\mu|(d\theta)\leq \kappa' \bar{K}(s-t), \quad s \leq t \leq T, \quad i \geq 0,
\end{align}
where $\bar K$ is given by \eqref{eq:kbarproof}.
Fix $i,j\geq n'$ and $t\leq T$. Combining all the above and integrating equation \eqref{eq:Delta_ijCauchy}  over the $\tau$ variable leads to
\bes{
    \left|\int_{\R_+} \Delta^{ij}_t(\theta,\tau)\mu(d\tau) \right| \leq c \int_t^T (1 + \bar{K}^2(s-t)) \left(\epsilon +  \left|\int_{\R_+} \Delta^{ij}_s(\theta,\tau)\mu(d\tau)\right|\right) ds, \quad \mu-a.e.
}
 An application of the generalized Gronwall inequality for convolution equation with $R$ the resolvent of $c(1 + \bar{K}^2)$, see \cite[Theorem 9.8.2]{GLS:90}, yields
\bes{
    \label{bound_delta_int_eps}
      \left|\int_{\R_+} \Delta^{ij}_t(\theta,\tau)\mu(d\tau) \right| \leq \epsilon c \left(T + \|\bar K\|^2_{L^2(0,T)}\right)\left(1 + \|R\|_{L^1(0,T)}\right),\quad  \mu-a.e.
}

{\textit{Step 2.}}
Plugging \eqref{bound:uniform_delta_ij}, \eqref{bound:int_intint}, \eqref{bound:Theta}, and \eqref{bound_delta_int_eps} into \eqref{eq:rho_ij}, we obtain
\bes{
    \label{ineq:Intrho_ij}
    \int_{\R_+}e^{-\theta(s-t)} \left| \rho^{ij}_s(\tau) \right| |\mu|(d\tau) \leq & r \left(\epsilon \left(1 +  \bar{K}(s-t)\right) +  \| \Delta^{ij}_s\|_{L^1(\mu \otimes \mu)} \right).
}
Finally by plugging \eqref{bound:int_intint}, \eqref{bound:Theta}, \eqref{bound_delta_int_eps} and \eqref{ineq:Intrho_ij}  into \eqref{eq:Delta_ijCauchy} and integrating over the $\theta$ and $\tau$ variables we obtain
\bes{
    \|\Delta^{ij}_t\|_{L^1(\mu \otimes \mu)} \leq &  c \int_t^T \left(1 + \bar K^2(t-s) \right)  \left(\epsilon + \|\Delta^{ij}_s\|_{L^1(\mu \otimes \mu)}  \right)ds.
}
 Another application of the generalized Gronwall inequality
 for convolution equations yields that 
\bes{
    \label{ineq:Delta_ij_eps}
    \|\Delta^{ij}_t\|_{L^1(\mu \otimes \mu)} \leq \epsilon c \left(T+ \|\Bar{K}^2\|^2_{L^2(0,T)}\right) \left(1 + \|{R}\|_{L^1(0,T)}\right).
}
This proves that $(\Gamma^i_t)_{i\geq 0}$ is a Cauchy sequence in  $L^1(\mu \otimes \mu)$ for every $t \in [0,T]$. 
\end{proof}

\subsection{Step 3: The limiting point of  $(\Gamma^i_t)_{i\geq 0}$ solves the Riccati equation}

\begin{lemma}
	\label{prop:existence_riccati}
	Assume  that  \eqref{assumption:QN} holds.  For each $t\leq T$, denote by   $\Gamma_t$ the limiting point in $L^1(\mu \otimes \mu)$ of   the  sequence $\left( \Gamma^i_t\right)_{i \geq 0}$ obtained from Lemma~\ref{L:cauchy}.  Then, $t\to \Gamma_t$ solves the Riccati equation \eqref{eq:Riccati_monotone_kernel_mild} with 
	\bes{
        \label{eq:bounded_riccati}
		\sup_{t\leq T}\|\Gamma_t\|_{L^1(\mu \otimes \mu)} < +\infty.
		}
\end{lemma}

\begin{proof}
Fix $t\leq T$.  By virtue of the $L^1(\mu\otimes \mu)$ convergence, 
\bes{
    \label{eq:convmumu}
    \Gamma^{i}_t(\theta, \tau) \to \Gamma_t(\theta, \tau) \quad \mu \otimes \mu-a.e. 
}
Furthermore the boundedness of $(i, t) \mapsto \left|\int_{\R_+^2} \mu(d\theta)^\top \Gamma^i_t(\theta,\tau)\mu(d\tau) \right|$,  $(i, t, \theta) \mapsto \left|\int_{\R_+} \Gamma_t^i(\theta,\tau)\mu(d\tau)\right|$ and  $(i, t, \tau) \mapsto \left|\int_{\R_+} \mu(d\theta)^\T\Gamma_t^i(\theta,\tau)\right|$, see Lemmas \ref{lemma:simple_convergence_lyapu} and \ref{lemma:simple_int_uniform_bound}, combined with equation \eqref{eq:recu_gamma_i} ensures that there exists a constant $c>0$ such that $|\Gamma^i_t(\theta,\tau)| \leq c \int_t^T e^{-(\theta + \tau){(s-t)}} ds  \leq  \left( 1 \vee T\right) \left(1 \wedge (\theta + \tau)^{-1} \right)$ since $1 - e^{-\theta t} \leq \left(1 \vee t \right)  \left( 1 \wedge \theta^{-1} \right)$. 
Hence the  dominated convergence  theorem yields   
\begin{align}
    &\int_{\R_+^2} \mu(d\theta)^\top \Gamma^i_t(\theta,\tau)\mu(d\tau) \to \int_{\R_+^2} \mu(d\theta)^\top \Gamma_t(\theta,\tau)\mu(d\tau), \label{eq:Ui} \\
     &\int_{\R_+^2} \mu(d\theta)^\top \Gamma^i_t(\theta, \tau) \to \int_{\R_+^2} \mu(d\theta)^\top \Gamma_t(\theta, \tau) \; \;\mbox{and} \; \;  \int_{\R_+^2} \Gamma^i_t(\theta,\tau)\mu(d\tau) \to \int_{\R_+^2} \Gamma_t(\theta,\tau)\mu(d\tau), \; \;    \mu-a.e. \nonumber
\end{align}
	Thus, as $i \to \infty$ we have 
	\bes{
	    \label{final_convergence}
		\Theta^i_t(\theta) & \to  \Theta_t(\theta) = \left(N + F^\T \int_{\R_+^2} \mu(d\theta)^\T \Gamma_t(\theta, \tau) \mu(d\tau) F\right)^{-1} \\
		& \qquad \qquad \qquad \times \left(F^\T \int_{\R_+^2}\mu(d\theta)^\T\Gamma_t{(\theta,\tau)} \mu(d\tau)   D + C^\T \int_{\R_+} \Gamma_t(\theta, \tau) \mu(d\tau)  \right)   \\
		B^i_t(\theta) & \to  B + C^\T \Theta_t(\theta)   \\
		D^i_t(\theta) & \to D + F^\T \Theta_t(\theta) 
	}
	By plugging these convergences into \eqref{eq:recu_gamma_i} we obtain that the limit $\Gamma$ solves
	\bes{
	    \label{eq:lyapunov_riccati}
		\Gamma_t(\theta, \tau) =& \int_t^T e^{-(\theta + \tau)(s-t)}  \widetilde{\mathcal R}_s(\theta,\tau) ds, \qquad \mu \otimes \mu-a.e.}
with 
\bes{\widetilde{\mathcal R}_t(\theta,\tau)=& 
		Q +    \Theta_t(\theta)N\Theta_t(\tau) + \int_{\R_+} \Gamma_t(\theta,\tau)\mu(d\tau) \left(B + C^\T \Theta_t(\tau)\right) \\
		&
		+ \left(B + C^\T \Theta_t(\theta)\right)^\T \int_{\R_+} \mu(d\theta)^\T \Gamma_s(\theta, \tau) \\
		&+ \left(D + F^\T \Theta_t(\theta)\right)^\T \int_{\R_+^2} \mu(d\theta)^\T \Gamma_t(\theta,\tau)\mu(d\tau) \left(D + F^\T \Theta_t(\tau)\right).
	\label{eq:widetildeR}}
	By using the expression of $\Theta$ exhibited in \eqref{final_convergence},  we get that 
	$\widetilde{\mathcal R}_t(\theta,\tau) = {\mathcal R}(\Gamma_t)(\theta,\tau),$ 
	where ${\mathcal R}$ is given by \eqref{eq:R1},  so that  \eqref{eq:lyapunov_riccati} is the desired Riccati equation.  Finally, the uniform bounds obtained in Lemmas \ref{lemma:simple_int_uniform_bound} and \ref{L:uniform_convergence_U} and plugged into \eqref{eq:lyapunov_riccati} imply \eqref{eq:bounded_riccati}.

 \end{proof}
  
 \subsection{Step 4: Continuity and uniqueness}

    We now establish the estimate \eqref{eq:estimateRiccati} for the solutions of the Riccati equation, which in turn implies continuity.

\begin{lemma}  \label{L:Riccatiestimate}
Assume that  there exists a $L^1(\mu\otimes \mu)$-valued function $t\mapsto \Gamma_{t}$ such that  \eqref{eq:Riccati_monotone_kernel_mild} holds with \eqref{eq:bounded_riccati}.
Then, the estimate \eqref{eq:estimateRiccati} holds and  $\Gamma \in C([0,T], L^1(\mu \otimes \mu))$. If in addition $Q\in \S^d_+$, then  $\Gamma_t \in \S^d_+(\mu\otimes \mu)$, for all $t\leq T$.
\end{lemma}

\begin{proof}
 The proof of the estimate  follows the same lines as that of Lemma \ref{L:lyapunovestimate}, with constant coefficients.  The only difference  is the nonlinear term 
	$$ \textbf{V}_t(\theta,\tau)=\int_t^T e^{-(\theta + \tau)(s-t)}   S(\Gamma_s)(\theta)^\top \hat{N}^{-1}(\Gamma_s)  S(\Gamma_s)(\tau) ds, $$
	 which we can bound as follows.   Let $\hat{S}(\Gamma)(s)(\theta) = |F||D| \|\Gamma_s\|_{L^1(\mu \otimes \mu)} +|C| \ints |\mu|(d \tau ')|\Gamma(\theta, \tau ')|$.  Integration over the $\tau$-variable, 	using the  bound $e^{-\theta(s-t)} \leq 1$ and  Tonelli's theorem give  for a constant   $c$  that may vary from line to line	 
	\begin{align*}
\int_{\R_+} |\mu|(d\tau)|\textbf{V}_t(\theta,\tau)|	&\leq  \int_{\R_+} |\mu|(d\tau) \int_t^T e^{-(\theta + \tau)(s-t)} \hat{S}(\Gamma)(s)(\theta) |N^{-1}|\hat{S}(\Gamma)(s)(\tau) ds \\
&\leq  c \sup_{s \leq T} \|\Gamma_s\|_{L^1(\mu \otimes \mu)}  \int_0^T (1+\bar K(s))  ds  \\
&\quad+ c \sup_{s \leq T} \|\Gamma_s\|_{L^1(\mu \otimes \mu)} \int_t^T (1 + \bar K(s-t)) \int_{\R_+} |\mu|(d\tau ') |\Gamma_s(\theta,\tau ')| ds,
	\end{align*}
	where $\bar K$ is defined  as in \eqref{eq:kbarproof}.	The first four terms appearing in 	$\int_{\R_+} |\mu|(d\tau) |\Gamma_t(\theta,\tau)|$ lead to inequality \eqref{eq:estimatetemplemma}, with $(\Gamma,\bar K,\mu,c)$ instead of $(\Psi,\bar K_2,\mu_2,\kappa)$. Adding the previous bound for the fifth nonlinear term yields
	\bes{
		\int_{\R_+} |\mu|(d\tau) |\Gamma_t(\theta,\tau)|  \leq & c \left(1 + \sup_{s \in [0,T]} \|\Gamma_s\|_{L^1(\mu_1 \otimes \mu)} \right) \int_0^T { \left(1+\bar{K}(s)\right)} ds\\
		&\quad + c \int_t^T (1+\bar {K}(s-t)) \int_{\R_+} |\mu|(d\tau ') |\Gamma_s(\theta,\tau ')| ds.
	}
The claimed estimate now follows from  the generalized Gronwall inequality for convolution equations, see \cite[Theorem 9.8.2]{GLS:90}. 

To argue continuity,  we recall that  the Riccati equation \eqref{eq:Riccati_monotone_kernel_mild} can be recast as a Lyapunov equation as in \eqref{eq:lyapunov_riccati}. The claimed continuity is  therefore a consequence of  Lemma \ref{l:lyapunov_continuity} provided that the coefficients of \eqref{eq:lyapunov_riccati} are bounded, which amounts to showing that the functions $ t \mapsto \int_{\R_+^2} \mu(d\theta)^\T \Gamma_t(\theta,\tau)\mu(d\tau)$
    and $ (t,\theta) \mapsto \int_{\R_+} \Gamma_t(\theta,\tau)\mu(d\tau)$ are bounded. The boundedness of the former is ensured by \eqref{eq:bounded_riccati} and that of the latter  follows from the estimate \eqref{eq:estimateRiccati}.  If in addition $Q\in \S^d_+$, then Lemma~\ref{lemma:positive_lyapunov} applied for \eqref{eq:lyapunov_riccati} yields that $\Gamma_t \in \S^d_+(\mu\otimes \mu)$ for any $t\leq T$. 
\end{proof}

Finally, exploiting once more the correspondence with the Lyapunov equation, uniqueness for the Riccati equation is obtained as a consequence of Theorem~\ref{thm:existence_unicite_lyapunov} and Lemma~\ref{lemma:positive_lyapunov}.

\begin{lemma}	\label{l:unicity_riccati}
There exists  at most one  solution to \eqref{eq:Riccati_monotone_kernel_mild}  such that \eqref{eq:estimateRiccati} and \eqref{eq:bounded_riccati} hold.
\end{lemma}

\begin{proof}
Let $\Gamma^a$ and $\Gamma^b$ be two solutions of \eqref{eq:Riccati_monotone_kernel_mild} such that \eqref{eq:estimateRiccati} and \eqref{eq:bounded_riccati} hold. 
For $i \in \{a,b\}$,  observe that $\Gamma^i$ can be recast as a solution to a Lyapunov equation  with bounded coefficients in the form \eqref{eq:recu_gamma_i}. 
As a result, $\Delta= \Gamma^a - \Gamma^b$ can be written as a solution to the following Lyapunov equation with bounded coefficients (see Lemma \ref{l:eq_gamma_i_recu} for details):
   \bec{
	\Delta^{}_t  (\theta, \tau) &= \int_t^T e^{-(\theta+\tau)(s-t)} F^{}_{ab}(s,\Delta^{}_s)  (\theta, \tau)ds,
	\\
	F^{}_{ab}(t,\Delta)  (\theta, \tau)&=Q^{ab}_t{   (\theta, \tau)}+ D_t^{b} (\theta)^\T \int_{\R_+^2} { \mu (d\theta')^\T \Delta(\theta',\tau') \mu(d\tau')} D_t^{b}(\tau) \\
	&\quad  + B_t^{b} (\theta)^\top \int_{\R_+} \mu (d\theta')^\T \Delta(\theta',\tau)  + \int_{\R_+} \Delta(\theta, \tau') \mu(d\tau') B_t^{2}(\tau) \\
	& \quad + S^{ab}_t (\theta)^\T \rho^{ab}_t(\tau) +\rho^{ab}_t (\theta)^\T  S^{ab}_t(\tau),    \label{eq:delta_ij_}
}
where
\begin{align*}  
    \rho^{ab}&=\Theta^{a} -\Theta^{b}, \\
    Q^{ab}_t   (\theta, \tau)&=\rho^{ab}_t (\theta)^\T \left(N + F^\T \int_{\R_+^2} \mu (d\theta')^\T   \Gamma^{a}_t(\theta',\tau')\mu(d\tau') F\right)\rho^{ab}_t(\tau),\\   
    S^{ab}_t(\tau) &= C^\T \int_{\R_+} \mu (d\theta')^\T \Gamma^a(\theta',\tau) + F^\T\int_{\R_+^2}\mu (d\theta') \Gamma^a_s(\theta',\tau') \mu(d\tau')D +\\
        & \;\;\;\;\;\;\;\;\;\; + \Big(N + F^\T \int_{\R_+^2} \mu (d\theta') \Gamma_s^a(\theta',\tau') \mu(d\tau') F \Big)\Theta^{a}_s(\tau). 
\end{align*}
The fact that the coefficients are bounded comes from  \eqref{eq:estimateRiccati} and \eqref{eq:bounded_riccati} on $\Gamma^a$ and $\Gamma^b$. Now, one can note 
similarly as in \eqref{eq:rho_ij} that $\rho^{ab}$ can be re-written as 
\bes{
    \rho^{ab}_s(\tau) =&-\left(N + F^\T \int_{\R_+^2} \mu(d\theta')^\T \Gamma^a_t(\theta',\tau ') \mu(d\tau ')   F\right)^{-1} \\
    &\quad\quad\quad\quad\quad\quad  \times \left(F^\T \int_{\R_+^2} \mu(d\theta')^\T \Delta_t(\theta',\tau ') \mu(d\tau ')  D +C^\T  \int_{\R_+} \mu(d\theta')^\T \Delta_t(\theta',\tau)  \right ) \\
    & + \left(  \left(N + F^\T \int_{\R_+^2} \mu(d\theta')^\T \Gamma^b_t(\theta',\tau ') \mu(d\tau ')   F\right)^{-1} - \left(N + F^\T \int_{\R_+^2} \mu(d\theta')^\T \Gamma^a_t(\theta',\tau ') \mu(d\tau ')   F\right)^{-1} \right) \\
    & \quad\quad\quad\quad\quad\quad  \times \left( F^\T \int_{\R_+^2} \mu(d\theta')^\T \Gamma^b_t(\theta',\tau ') \mu(d\tau ')  D +C^\T \int_{\R_+} \mu(d\theta')^\T \Gamma^b_t(\theta',\tau)  \right) \\
    =& \; \textbf{A}(\tau) + \textbf{B}(\tau), 
}
{which is linear in  $\Delta$ since  $\textbf{B}(\tau)$ can be rewritten as}   
\bes{
	& \left|  \left(N + F^\T \int_{\R_+^2} \mu(d\theta')^\T \Gamma^b_t(\theta',\tau ') \mu(d\tau ')   F\right)^{-1} - \left(N + F^\T \int_{\R_+^2} \mu(d\theta')^\T \Gamma^a_t(\theta',\tau ') \mu(d\tau ')   	F\right)^{-1} \right|  \\
	&= \Bigg | \left(N + F^\T \int_{\R_+^2} \mu(d\theta')^\T \Gamma^b_t(\theta',\tau ') \mu(d\tau ')   F\right)^{-1}    \left(F^\T \int_{\R_+^2} \mu(d\theta')^\T \left(\Delta_t(\theta',\tau ') \right)  \mu(d\tau ') F  \right)  \\
	 & \qquad \times  \left(N + F^\T \int_{\R_+^2} \mu(d\theta')^\T \Gamma^a_t(\theta',\tau ') \mu(d\tau ')   F\right)^{-1} \Bigg |. 
}
Consequently, $\Delta = \Gamma^a - \Gamma^b $ is solution to a homogeneous linear Lyapunov equation with bounded coefficients, and no affine term. Thus, the generalized Gronwall inequality for convolution equations, see \cite[Theorem 9.8.2]{GLS:90} ensures that $\| \Delta_t \|_{L^1(\mu \otimes \mu)}=0$ for every $t\in [0,T]$, which proves uniqueness. 
\end{proof}


\vspace{1mm}

\appendix

\section{Some elementary results}


\begin{lemma}
\label{l:CS}
Let $\Psi \in \S^{d}_+(\mu \otimes \mu)$, and $\boldsymbol{\Psi}$ its corresponding linear integral operator. Then for any $\varphi,\psi \in L^1(\mu)$
\begin{align*}
\langle \varphi,\boldsymbol{\Psi} \psi\rangle^2_{\mu} & \leq \; \langle \varphi,\boldsymbol{\Psi} \varphi \rangle_{\mu}. \langle \psi,\boldsymbol{\Psi} \psi\rangle_{\mu} 
\end{align*}
\end{lemma}
\begin{proof}
	Since $\Psi \in \S^{d}_+(\mu \otimes \mu)$, then for any $\varphi, \psi\in  L^1(\mu)$ and $\lambda \in \R$ we have 
	\bes{
		 \int_{\R_+^2} \left( \varphi(\theta) + \lambda \psi(\theta)\right)^\T \mu(d \theta)^\T \Psi(\theta, \tau) \mu(d\tau) \left( \varphi(\tau) + \lambda \psi(\tau) \right) \geq 0.
	}
	By expanding the square we obtain a non negative second order polynomial in $\lambda$ whose discriminant must be non positive. This combined with $\Psi(\theta, \tau) = \Psi(\tau, \theta)^\T $ yields the claimed inequality. 
\end{proof}

\begin{lemma}
\label{l:eq_gamma_i_recu}
Let $(\Gamma^i)_{i\geq 0}$ be the sequence defined in \eqref{eq:recu_gamma_i}. Then for any $1 \leq i < j$, $\Delta^{ij} = \Gamma^i - \Gamma^j$ is solution to
   \bec{
	\Delta^{ij}_t  (\theta, \tau) &= \int_t^T e^{-(\theta+\tau)(s-t)} F^{\delta}_{ij}(s,\Delta^{ij}_s)  (\theta, \tau)ds,
	\\
	F^{\delta}_{ij}(t,\Delta)  (\theta, \tau)&=Q^{ij,\delta}_t{   (\theta, \tau)}+ D_t^{j-1} (\theta)^\T \int_{\R_+^2} { \mu (d\theta')^\T \Delta(\theta',\tau') \mu(d\tau')} D_t^{j-1}(\tau) \\
	&\quad  + B_t^{j-1} (\theta)^\top \int_{\R_+} \mu (d\theta')^\T \Delta(\theta',\tau)  + \int_{\R_+} \Delta(\theta, \tau') \mu(d\tau') B_t^{j-1}(\tau) \\
	& \quad + S^{ij}_t (\theta)^\T \rho^{ij}_t(\tau) +\rho^{ij}_t (\theta)^\T  S^{ij}_t(\tau),    \label{eq:delta_ij_}
}
where
\begin{align*}  
    \rho^{ij}&=\Theta^{i-1} -\Theta^{j-1}, \\
    Q^{ij,\delta}_t   (\theta, \tau)&=\rho^{ij}_t (\theta)^\T \left(N + F^\T \int_{\R_+^2} \mu (d\theta')^\T   \Gamma^{i}_t(\theta',\tau')\mu(d\tau') F\right)\rho^{ij}_t(\tau),\\   
    S^{ij}_t(\tau) &= C^\T \int_{\R_+} \mu (d\theta')^\T \Gamma^i(\theta',\tau) + F^\T\int_{\R_+^2}\mu (d\theta') \Gamma^i_s(\theta',\tau') \mu(d\tau')D +\\
        & \;\;\;\;\;\;\;\;\;\; + \Big(N + F^\T \int_{\R_+^2} \mu (d\theta')^\T \Gamma_s^i(\theta',\tau') \mu(d\tau') F \Big)\Theta^{j-1}_s(\tau). 
\end{align*}
\end{lemma}

\begin{remark}\label{R:deltai}
    Note that when $j = i+1$, then $S^{i(i+1)}=0$. Indeed, in such case we have  
    \bes{
        S^{i(i+1)}_s (\theta) = &C^\T  \int_{\R_+} \mu (d\theta') \Gamma^{i}_s(\theta',\tau) + F^\T \int_{\R_+^2} \mu (d\theta')^\T \Gamma^i_s(\theta',\tau') D^{j-1}_s(\tau) + N \Theta^{i}_s(\tau)  \\
         =& C^\T  \int_{\R_+} \mu (d\theta') \Gamma^{i}_s(\theta',\tau) + F^\T \int_{\R_+^2} \mu (d\theta')^\T \Gamma^i_s(\theta',\tau') D \\
        &+ \left(N + F^\T \int_{\R_+^2} \mu (d\theta')^\T \Gamma^i_s(\theta',\tau') \mu(d\tau') F \right) \Theta^{i}_s(\tau)    \;    = \;  0.
    }
 As a consequence, in the particular case where $j=i+1$, $\Delta^i = \Delta^{i(i+1)}$ is solution to \eqref{eq:Deltai}. \qed
\end{remark}

\begin{proof}
Let $t\in [0,T]$, for almost every  $\theta,\tau$ we have
    \bes{
        \label{eq:proof_delta_ij}
        \Delta^{ij}_t  (\theta, \tau) =& \Gamma^i_t  (\theta, \tau) - \Gamma^j_t  (\theta, \tau) \\
        =& \int_t^T  e^{-(\theta+\tau)(s-t)} \left(\textbf{I}^{ij,\delta}_s  (\theta, \tau) + \left(\textbf{I}^{ij,\delta}_s(\tau,\theta)\right)^\T + \textbf{II}^{ij,\delta}_s  (\theta, \tau)+ \textbf{III}^{ij,\delta}_s  (\theta, \tau) \right)ds ,
    }
where $\textbf{I}^{ij,\delta}, \textbf{II}^{ij,\delta}$  and $\textbf{III}^{ij,\delta}$ are defined as follows
\bes{
        { \textbf{I}^{ij,\delta}_s  (\theta, \tau) } =& \int_{\R_+} \Gamma^i_s(\theta, \tau') \mu(d\tau') B^{i-1}_s(\tau) - \int_{\R_+} \Gamma^j_s(\theta, \tau') \mu(d\tau') B^{j-1}_s(\tau)\\
        =& \int_{\R_+} \Delta^{ij}_s(\theta, \tau')\mu(d\tau')B^{j-1}_s(\tau) + \int_{\R_+} \Gamma^{i}_s(\theta, \tau')\mu(d\tau')(B^{i-1}_s(\tau) - B^{j-1}_s(\tau))   \\ 
            =& \int_{\R_+} \Delta^{ij}_s(\theta, \tau')\mu(d\tau') B^{j-1}_s(\tau) + \int_{\R_+} \Gamma^{i}_s(\theta, \tau')\mu(d\tau')C \rho^{ij}_s(\tau) \\
        { \textbf{II}^{ij,\delta}_s  (\theta, \tau)} =& D^{i-1}_s (\theta)^\T \int_{\R_+^2} \mu (d\theta')^\T \Gamma^i_s  (\theta, \tau) \mu(d\tau')D^{i-1}_s(\tau) - D^{j-1}_s (\theta)^\T \int_{\R_+^2} \mu (d\theta')^\T \Gamma^j_s  (\theta, \tau) \mu(d\tau')D^{j-1}_s(\tau)  \\
        =& D^{j-1}_s (\theta)^\T \int_{\R_+^2} \mu (d\theta')^\T \Delta^{ij}_s  (\theta, \tau) \mu(d\tau')D^{j-1}_s(\tau) \\
        & + \rho^{ij}_s (\theta)^\T F^\T \int_{\R_+^2} \mu (d\theta')^\T \Gamma^i_s(\theta',\tau') \mu(d\tau') F \rho^{ij}_s(\tau) \\
        & +  \rho^{ij}_s (\theta)^\T F^\T \int_{\R_+^2} \mu (d\theta')^\T \Gamma^i_s(\theta',\tau') D^{j-1}_s(\tau) \\
        & + D^{j-1}_s (\theta)^\T \int_{\R_+^2} \mu (d\theta')^\T \Gamma^i_s(\theta',\tau') \mu(d\tau') F \rho^{ij}_s(\tau) \\
        { \textbf{III}^{ij,\delta}_s  (\theta, \tau)} =&\;   Q^{i-1}_s  (\theta, \tau) - Q^{j-1}_s  (\theta, \tau)\\
        =& \; \rho^{ij}_s (\theta)^\T N \rho^{ij}_s(\tau) + \rho^{ij}_s (\theta)^\T N \Theta^{j-1}_s(\tau) + \Theta^{j-1}_s (\theta) N \rho^{ij}_s(\tau)
    }
 By  plugging the expressions of $\textbf{I}^{ij,\delta}, \textbf{II}^{ij,\delta}, \textbf{III}^{ij,\delta}$ into \eqref{eq:proof_delta_ij}  we obtain    \eqref{eq:delta_ij_}. 
\end{proof}

\vspace{5mm}

\small



\bibliographystyle{abbrv}

\bibliography{bibl}

\typeout{get arXiv to do 4 passes: Label(s) may have changed. Rerun}

\end{document}

%% file: functions.tex
\newcommand{\dd}{\mathrm{d}}
\newcommand{\proofspace}{\mbox{} \\*}
\newcommand{\entrecro}[1]{\left[ #1 \right]}
\newcommand{\entrebra}[1]{\left \{  #1 \right \}}
\newcommand{\entrepar}[1]{\left(  #1 \right) }
\renewcommand{\bar}{\overline}
\newcommand{\1}{\mathbf{1}} 
\newcommand{\oo}[1]{ {\fontsize{6}{6}\selectfont \textcircled{\raisebox{-0.79pt}{\hspace*{0.001pt} #1} }}  } 
\newcommand{\be}[1]{\begin{equation} #1 \end{equation}}
\newcommand{\bec}[1]{\begin{equation} \begin{cases} #1\end{cases} \end{equation}}
\newcommand{\bes}[1]{\begin{equation} \begin{split} #1\end{split} \end{equation}}
\newcommand{\B}[1]{\boldsymbol{#1}}
\newcommand{\cali}[1]{ \mathcal{#1}}
\newcommand{\produit}{\bigstar}

\newcommand{\Chi}{\scalebox{1.1}{$\chi$}}
\newcommand{\Pg}{P^g_t}
\newcommand{\Px}{P^x_t}
\newcommand{\PI}[1]{P^i_t(#1)}
\newcommand{\Pc}{P^c_t}

\newcommand{\inter}{\llbracket 1,n \rrbracket}
\newcommand{\actionspace}{\mathbf{A}}
\newcommand{\lambdaspace}[1][]{%
\ifthenelse{\equal{#1}{}}{C_\lambda([0,T], L^1(\mu_1 \otimes \mu_2))}{C_{#1}([0,T], L^1(\mu_1 \otimes \mu_2))}%
}
\newcommand{\var}{\mathrm{Var}}
\renewcommand{\P}{\mathbb{P}}

\newcommand{\scal}[2]{#1.#2^{\otimes2}}

\renewcommand{\c}{\alpha}

%% file: RiccatiVolterraFinal.bbl
\begin{thebibliography}{10}

\bibitem{AJwishart}
E.~Abi~Jaber.
\newblock The {L}aplace transform of the integrated {V}olterra {W}ishart
  process.
\newblock In preparation.

\bibitem{abietal19a}
E.~Abi~Jaber, E.~Miller, and H.~Pham.
\newblock Linear--{Q}uadratic control for a class of stochastic {V}olterra
  equations: solvability and approximation.
\newblock 2019.

\bibitem{alfonsi2013capacitary}
A.~Alfonsi and A.~Schied.
\newblock Capacitary measures for completely monotone kernels via singular
  control.
\newblock {\em SIAM Journal on Control and Optimization}, 51(2):1758--1780,
  2013.

\bibitem{arta}
N.~Artamonov.
\newblock Solvability of an operator {R}iccati integral equation in a reflexive
  banach space.
\newblock {\em Differential Equations}, 55(5):718--728, 2019.

\bibitem{brezis2010functional}
H.~Brezis.
\newblock {\em Functional analysis, Sobolev spaces and partial differential
  equations}.
\newblock Springer Science \& Business Media, 2010.

\bibitem{curtain}
R.~Curtain and A.~Pritchard.
\newblock The infinite-dimensional {R}iccati equation.
\newblock {\em Journal of Mathematical Analysis and Applications}, 47:43--57,
  1974.

\bibitem{daprato}
G.~Da~Prato.
\newblock Direct solution of a {R}iccati equation arising in stochastic control
  theory.
\newblock {\em Applied Mathematics and Optimization}, 11:191--208, 1984.

\bibitem{flandoli1986direct}
F.~Flandoli.
\newblock Direct solution of a {R}iccati equation arising in a stochastic
  control problem with control and observation on the boundary.
\newblock {\em Applied Mathematics and Optimization}, 14(1):107--129, 1986.

\bibitem{GLS:90}
G.~Gripenberg, S.-O. Londen, and O.~Staffans.
\newblock {\em Volterra integral and functional equations}, volume~34 of {\em
  Encyclopedia of Mathematics and its Applications}.
\newblock Cambridge University Press, Cambridge, 1990.

\bibitem{guates05}
G.~Guatteri and G.~Tessitore.
\newblock On the backward stochastic {R}iccati equation in infinite dimension.
\newblock {\em SIAM Journal on Control and Optimization}, 44(1):159--194, 2005.

\bibitem{hu2018stochastic}
Y.~Hu and S.~Tang.
\newblock Stochastic {LQ} and associated {R}iccati equation of {PDE}s driven by
  state-and control-dependent {W}hite noise.
\newblock {\em arXiv preprint arXiv:1809.05308}, 2018.

\bibitem{kos}
S.~Koshkin.
\newblock Positive semigroup and algebraic {R}iccati equations in {B}anach
  spaces.
\newblock {\em Positivity}, 20:541--563, 2016.

\bibitem{lasiecka}
I.~Lasiecka.
\newblock Optimal control problems and {R}iccati equations for systems with
  unbounded controls and partially analytic generators.
\newblock {\em Functional Analytic Methods for Evolution Equations}, pages
  313--371, 2005.

\bibitem{yong1999stochastic}
J.~Yong and X.~Y. Zhou.
\newblock {\em Stochastic controls: Hamiltonian systems and HJB equations},
  volume~43.
\newblock Springer Verlag, 1999.

\end{thebibliography}
